\documentclass[11pt]{article}
\usepackage{amsmath,amssymb,mathrsfs}
\usepackage{color}

\usepackage[colorlinks,citecolor=blue,urlcolor=blue]{hyperref}

\newtheorem{Theorem}{Theorem}[part]

\newtheorem{Assumption}{Assumption}[part]

\newtheorem{Corollary}{Corollary}[part]
\newtheorem{Remark}{Remark}[part]

\DeclareMathOperator*{\essinf}{ess\,inf}

\DeclareMathOperator*{\esssup}{ess\,sup}

\def\cal#1{\mathcal{#1}}

\def\v{\varphi}

\def\b{\beta}

\def\L{\Lambda}

\def\reff#1{{\rm(\ref{#1})}}

\def\Gc{{\cal G}}

\def\L{{\cal L}}

\def\T{{\cal T}}

\def\Dzw1#1{\frac{\partial^2 #1}{\partial z \partial w_1}}

\def\Dzb1#1{\frac{\partial^2 #1}{\partial z \partial b_1}}

\newcommand{\dproof}{\noindent {Proof.} \quad}
\newcommand{\fproof}{\hfill $\square$ \bigskip}

\newtheorem{definition}{Definition}[section]

\newtheorem{theorem}[definition]{Theorem}

\newtheorem{proposition}[definition]{Proposition}
\newtheorem{lemma}[definition]{Lemma}

\def\1B{\text{1\!\!I}}

\def\FB{\mathbb{F}}

\def\PB{\mathbb{P}}

\def\NB{\mathbb{N}}

\def\1B{\text{1\!\!I}}

\def\reff#1{{\rm(\ref{#1})}}

\def \be{\begin{eqnarray}}
\def \ee{\end{eqnarray}}
\def \b*{\begin{eqnarray*}}
\def \e*{\end{eqnarray*}}

\def\P{\mathbb{P}}

\begin{document}

\title{A new Mertens decomposition of $\mathscr{Y}^{g,\xi}$-submartingale systems. Application to BSDEs with weak constraints at stopping times  \footnote{This work was conducted while the third author was visiting the National University of Singapore, whose hospitality is kindly acknowledged. The forth author also acknowledges the hospitality of Le Mans University during his visit.} }
\author{Roxana Dumitrescu
\thanks{Department of Mathematics, King's College London, United Kingdom, email: {\tt roxana.dumitrescu@kcl.ac.uk}.}
\and
Romuald Elie
\thanks{LAMA, 
Universit\'e Gustave Eiffel, CNRS, France, 
email: {\tt romuald.elie@univ-mlv.fr }}
\and Wissal Sabbagh\thanks{Laboratoire Manceau de Math\'{e}matiques, Institut du Risque et de l'Assurance, Le Mans Universit\'e, France. email: {\tt wissal.sabbagh@univ-lemans.fr}. The author's research is part of the ANR project DREAMeS (ANR-21-CE46-0002).} \and Chao Zhou \thanks{Department of Mathematics, National University of Singapore, Singapore,  {\tt matzc@nus.edu.sg.} Research supported by Singapore MOE AcRF Grants R-146-000-219-112 and R-146-000-255-114.}}

\maketitle
\begin{abstract}
We first introduce the concept of $\mathscr{Y}^{g,\xi}$-submartingale systems, where the nonlinear operator $\mathscr{Y}^{g,\xi}$ corresponds to the first component of the solution of a reflected BSDE with generator $g$ and lower obstacle $\xi$. We first show that, in the case of a left-limited right-continuous obstacle, any $\mathscr{Y}^{g,\xi}$-submartingale system can be aggregated by a process which is right-lower semicontinuous. We then prove a \textit{Mertens decomposition}, by using an original approach which does not make use of the standard penalization technique. These results are in particular useful for the treatment of control/stopping game problems and, to the best of our knowledge, they are completely new in the literature. As an application, we introduce a new class of \textit{Backward Stochastic Differential Equations (in short BSDEs) with weak constraints at stopping times}, which are related to the partial hedging of American options. We study the wellposedness of such equations and, using the $\mathscr{Y}^{g,\xi}$-Mertens decomposition, we show that the family of minimal time-$t$-values $Y_t$, with $(Y,Z)$ a supersolution of the BSDE with weak constraints, admits a representation in terms of a reflected backward stochastic differential equation.


\end{abstract}

\vspace{10mm}

\noindent{\bf Key words~:}  Mertens decomposition, BSDEs with weak constraints at stopping times, 
Optimal control, Optimal stopping, Stochastic game, Stochastic target.

\vspace{10mm}

\noindent{\bf AMS 1991 subject classifications~:} 93E20, 60J60, 47N10.


\section{Introduction}

\noindent The Doob-Meyer decomposition represents a fundamental tool of the general theory of processes,  in particular when applied to optimal control. This result, first introduced in \cite{M72} for classical supermatingales, has been further extended in \cite{P97} to a class of right-continuous $g$-supermatingales, being used by many authors in various contexts: backward stochastic differential equation with constraints \cite{BEM18,KLIMSIAK15,PX10}, minimal supersolutions under non-classical conditions on the driver \cite{DHK13,HKM14}, backward stochastic differential equations with weak terminal conditions \cite{BER15}.  More recently, a version \`a la Mertens of the Doob-Meyer decomposition has been provided in \cite{BPT16} and \cite{GIOOQ171} for $g$-supermatingales, which are not necessarily right-continuous, a natural application of their decomposition being to the general duality for the minimal super-solution of a BSDEs with constraint. 

In this paper, we introduce the notion of $\mathscr{Y}^{g,\xi}$-submartingale processes, with  $\mathscr{Y}^{g,\xi}$ the nonlinear operator associated to the first component of the solution of a reflected backward stochastic differential equation with lower obstacle $\xi$ and driver $g$, and provide a Doob-Meyer-Mertens decomposition for such processes. To the best of our knowledge, this result is completely new  in the literature, even in the classical case of a driver $g \equiv 0$.  To illustrate the main result, we consider the simpler setting of right-continuous left-limited strong  $\mathscr{Y}^{g,\xi}$-submartingale processes, the associated $\mathscr{Y}^{g,\xi}$-Doob-Meyer decomposition being new as well.   More precisely, let $(\Omega, \mathscr{F}, \mathbb{P})$ be a probability space, equipped with a {\color{black}{d-dimensional}} Brownian motion $W$, as well as the Brownian filtration $\mathbb{F}:=(\mathscr{F}_t)_{0 \leq t \leq T}$. Let $g:[0,T] \times \Omega \times \mathbf{R} {\color{black}{\times \mathbf{R}^d}} \mapsto \mathbf{R}$ be some function, progressively measurable in $(t,\omega)$ and Lipschitz in $(y,z)$, also called \textit{driver}, $\xi$ a right-continuous left-limited process  and $\eta \in \textbf{L}_2(\mathscr{F}_\tau)$, with $\eta \geq \xi_\tau$, where $\tau$ is a stopping time. We define $\mathscr{Y}_{\cdot,\tau}^{g,\xi}[\eta]:=Y_\cdot$, where $(Y,Z,A)$ is the unique solution of the reflected backward stochastic differential equation
\begin{align*}
-dY_t=g(t,Y_t,Z_t)dt-Z_tdW_t+dA_t \text{\,\,on\,\,} [0,\tau]; \nonumber\\
Y_t \geq \xi_t; \nonumber\\
{\color{black}{Y_{\tau} =\eta}};\nonumber\\
\int_0^T(Y_{t^-}-\xi_{t^-})dA_t=0 \text{\, a.s.}
\end{align*}
Then, an optional process $\textbf{X}$ is called a \textit{strong} $\mathscr{Y}^{g,\xi}$-\textit{submartingale} process if $\textbf{X}_\sigma \geq \xi_\sigma$, for all stopping times $\sigma$ and such that $\textbf{X}_S \leq \mathscr{Y}_{S,\tau}^{g,\xi}[\mathbf{X}_\tau]$ a.s. on $S \leq \tau$, for all stopping times $S, \tau$. When the process $\textbf{X}$ is right-continuous, we show that it admits the unique Doob-Meyer decomposition 
\begin{align}\label{Doob}
-d\textbf{X}_t=g(t,\textbf{X}_t,Z^X_t)dt-Z^X_tdW_t+dA^X_t-dK^X_t \text{\,\,on\,\,} [0,\tau]; \nonumber\\
\textbf{X}_t \geq \xi_t; \nonumber\\
\int_0^T(\textbf{X}_{t^-}-\xi_{t^-})dA_t=0, \text{\, a.s.}
\end{align}
where $A^X$ and $K^X$ are {\color{black}{non-decreasing}} right-continuous left-limited processes such that $A^X_0=0$ and $K^X_0=0$.

We give a general regularity result for strong $\mathscr{Y}^{g,\xi}$-submartingale processes, in particular we show that they are right-lower semicontinuous, and provide a \textit{Mertens version} of the $\mathscr{Y}^{g,\xi}$-Doob-Meyer decomposition $\eqref{Doob}$. Our approach is original, avoiding the standard penalization technique adopted in the context of classical $\mathcal{E}^g$-supermartingales. In particular, it is important to notice that, to get the decomposition, we do not require the right-continuity, which might be difficult to prove, the right-lower semicontinuity which comes from the $\mathscr{Y}^{g,\xi}$-submartingale property being enough. Our method also allows to obtain the existence of left and right limits of $\mathscr{Y}^{g,\xi}$-submartingales, without the need of establishing a down-crossing inequality which is the usual approach for $\mathcal{E}^g$-supermartingales (see e.g. \cite{peng:99}, \cite{BPT16}). Furthermore, our results are given in the more general setting of $\mathscr{Y}^{g,\xi}$-submartingale systems, which naturally appear in control theory. Using tools from the general theory of processes, we show that such a system can be aggregated by a right-lower semicontinuous optional process, which admits the  \textit{\`a la Mertens version} of the Doob-Meyer decomposition. Such a decomposition seems as fundamental as the Mertens decomposition of standard $\mathcal{E}^g$-supermartingales, being useful in various contexts, in particular for the treatment of control-stopping game problems.\\

We illustrate an application of our decomposition to the study of a new class of \textit{BSDEs with weak constraints} at any stopping time, which are related to the partial hedging of American options. This class of BSDEs extends the so-called \textit{BSDEs with weak terminal condition} introduced in \cite{BER15}, by allowing for weak constraints at any stopping time. The Doob-Meyer-Mertens decomposition of $\mathscr{Y}^{g,\xi}$-submartingales is a key ingredient for the dynamic characterization of the family of minimal time $\tau$ values $Y_\tau$ such that $(Y,Z)$ is a supersolution of the \textit{BSDE with weak constraints}, that is it satisfies, for $0 \leq t \leq T$, 
\begin{align}\label{Dynamic_1}
Y_t^Z \geq Y_0-\int_0^t g(s,Y_s^Z,Z_s)ds+\int_0^tZ_sdW_s; 
\end{align}
\begin{align}\label{weak_cond_1}
\mathbb{E}[\ell(Y_{\tau}^{Z}-L_{\tau})]\geq m, \,\,\, \PB-a.s.\,\text{for all  stopping time } \tau \text{ taking values in }[0,T],
\end{align}
{\color{black}where $m$ is a given threshold and $\ell$ is a nondecreasing map.}\\
The problem at time $0$ reads as follows:
\begin{align}\label{prob}
    \text{Find the minimal } Y_0 \text{ such that } \eqref{Dynamic_1} \text{ and } \eqref{weak_cond_1} \text{ are satisfied for some } Z.
\end{align}
From a practical point of view, the cost of superhedging is fairly  too high, so that the option seller needs to accept to take some risk. An approach, which has been mainly developed for European options, consists in replacing the too strong super-replicating $\P$-a.s. condition by a weaker one (see e.g. \cite{BET10}, \cite{BER15}, \cite{D16}). By analogy, in the case  of American options, the seller needs to solve a BSDE with dynamics \eqref{Dynamic_1}, but shall replace the too strong constraint $Y_\tau \geq L_\tau \P$-a.s.\,\text{for all  stopping time }  $\tau$ \text{ taking values in }$[0,T]$  by the weaker \eqref{weak_cond_1}, where $m$ stands for a given success threshold and $\ell$ represents a non-decreasing loss function. In contrast with the case of European Options, the literature on the partial hedging of American options and Bermudean type options is much less abundant. Related works are \cite{BBC16} (where the authors propose  a probabilistic numerical algorithm for the computation of the quantile hedging of Bermudean options) and \cite{BEH16} ( which follows a very different approach to study BSDEs of the form \reff{Dynamic_1} together with a weaker version of \reff{weak_cond_1} where the constraint only holds for deterministic times valued in $[0,T]$). In the framework of the latter paper, no dynamic programming principle is available and the derived solution connects to stochastic differential equations of McKean-Vlasov type.

The first step in our analysis consists in the reformulation of \eqref{prob}
as a control-stopping problem. To this purpose, we show first that if $Y$ and $Z$ are such that $\eqref{weak_cond_1}$ is satisfied, then we can find an admissible control process $\alpha$ such that
\begin{align}\label{strong}
\Psi(\theta, Y_\theta) \geq M_\theta^\alpha:=m+\int_0^\theta \alpha_s {\color{black}dW_s}, \text{\,\, for all stopping times\,\,} \theta,
\end{align}
{\color{black}where $m\in\mathbb R$ and $\Psi$ is a (possibly random) nondecreasing real-valued map\footnote{In fact, we can take $\Psi(\tau, Y_\tau)= \ell(Y_{\tau}^{Z}-L_{\tau})$.}.}\\
In the context of a terminal constraint only, the existence of such a control process $\alpha$ is simply obtained by the Martingale Representation Theorem (see e.g. \cite{BER15}). In our setting, since we have to deal with weak constraints at any stopping time, the existence of a control such that $\eqref{strong}$ holds  is not obvious and more sophisticated arguments are needed to prove it. Using this result, we show that the solution of \eqref{prob} can be written as
\begin{align}\label{control-stopping}
    \inf_\alpha \sup_\theta \mathcal{E}_{0,\theta}^{g}[\Phi(\theta,M_\theta^\alpha)],
\end{align}
{\color{black}with $\Phi$ the left-continuous inverse of the non-decreasing map $\Psi$.}
Equivalently, using the link between the solution of reflected BSDEs and  nonlinear optimal stopping with $g$-expectations, the above control-stopping problem can be rewritten under the form
\begin{align}
    \inf_\alpha \mathscr{Y}_{0,T}^{g, \alpha}[\Phi(T,M_T^\alpha)],
\end{align}
where the nonlinear operator $\mathscr{Y}^{g, \alpha}$ is associated to the obstacle process $\Phi(\cdot, M^\alpha_\cdot)$.

Our aim is to study, for a threshold process $M^\alpha_\cdot$, the family of minimal $\tau$-values $Y_\tau$, which is given by
\begin{align}
\mathcal{Y}^\alpha(\tau)= \essinf \{Y_\tau^{\alpha'}, \,\, \alpha'=\alpha  {\rm \,\, on \,\,} [\![0,\tau]\!]\}.
\end{align}
We  derive a dynamic programming principle for this family, from which we deduce that  $(\mathcal{Y}^\alpha(\tau))_{\tau}$ is a $\mathscr{Y}^{g,\alpha}$-submartingale system which can be aggregated by a right-continuous with left limits process $\mathcal{Y}_\cdot^\alpha$.  Taking advantage of the  $\mathscr{Y}^{g,\xi}$-Mertens-Doob-Meyer decomposition proved in the first part of the paper, we show that the value process $\mathcal{Y}_\cdot^\alpha$ corresponds to the unique solution of a specific \textit{reflected backward stochastic differential equation}. 
Finally, we provide some complementary results related to the control-stopping game problem $\eqref{control-stopping}$, and give some sufficient conditions under which the game problem admits a value and a saddle-point.

\vspace{3mm}

\noindent The outline of the paper is the following: In Section \ref{decomp}, we introduce the $\mathscr{Y}^{g,\xi}$-submartingale systems, show that they can be aggregated by right-lower semicontinuous processes and provide the $\mathscr{Y}^{g,\xi}$-Mertens-Doob-Meyer decomposition. In Section \ref{BSDE}, we illustrate an application of this decomposition to the study of a new class of \textit{BSDEs with weak constraints at stopping times}. In particular, we show that the family of minimal time $t$-values $Y_t$, with $(Y,Z)$ a supersolution, is a strong $\mathscr{Y}^{g,\xi}$-submartingale, which can be characterized as the unique solution of a specific reflected backward stochastic differential equation. Finally, in Appendix \ref{Sec4}, we give some sufficient conditions under which the control-stopping problem \eqref{control-stopping} admits a value and a saddle point.

%
%
%
%


\paragraph{Notations}
\label{notation}
We first introduce a series of notations that will be used throughout the paper. Let $d\geq 1$ and $T>0$ be fixed. We denote by $W:=(W_t)_{t\in[0,T]}$ a $d$-dimentional Brownian motion defined on a probability space $(\Omega,\mathscr{F},\PB)$ with $\PB$- augmented natural filtration $\FB= (\mathscr{F}_t)_{t\in[0,T]}$. The notation $\mathbb{E}$ will stand for the expectation with respect to $\PB$.\\
{\color{black} Let $\mathcal P$ be the predictable $\sigma$-algebra on $\Omega\times [0,T]$. For each $T>0$, we introduce the following spaces:}
\begin{itemize}
\item
 $\textbf{L}_p(U,{\Gc})$  is the set of $p$-integrable $\Gc$-measurable random variables with values in $U$, $p\geq 0$, $U$ a Borel set of $\mathbf{R}^n$ for some $n\geq 1$ and $\Gc \subset \mathscr{F}$. When $U$ and $\Gc$ can be clearly identified by the context, we omit them. This will be in particular the case when $\Gc=\mathscr{F}$.
\item    $\mathbf{H}_2$ is the set of
$\mathbf{R}^d$-valued $\FB$-predictable processes $\phi=(\phi_t)_{t\in[0,T]}$ such that $$\| \phi\|^2_{\mathbf{H}_2} := \mathbb{E} \left[\int_0 ^T |\phi_t |^2 dt \right] < \infty.$$

\item ${\mathbf K}_2$ is the set of real-valued non-decreasing RCLL and $\FB$-predictable
 processes $K=(K_t)_{t\in[0,T]}$ with $K_0 = 0$ and $\mathbb{E}[K^2_T] < \infty$. 

\end {itemize}

\begin{itemize}

\item  $\T_{0}$ denotes the set of $\FB$-stopping times $\tau$ such that $\tau \in [0,T]$ a.s. The notation $\mathbb{E}_\tau[.]$ stands for the conditional expectation given $\mathscr{F}_\tau$, $\tau \in \T_{0}$.

\item For $\theta$ in $\T_{0}$,    $\T_{\theta}$  is the set of stopping times
$\tau\in \T_{0}$ such that $\theta \leq \tau \leq T$  $\PB$-a.s.

\item   ${\mathbf S}_{2}$ is the set of real-valued optional 
 processes $\phi=(\phi_t)_{t\in[0,T]}$ such that
$$\| \phi\|^2_{{\mathbf S}_{2}} := \mathbb{E}[\esssup_{\tau \in \mathcal{T}_0} |\phi_\tau |^2]<  \infty.$$

\end {itemize}

\section{$\mathscr{Y}^{g,\xi}$-Mertens-Doob-Meyer decomposition of $\mathscr{Y}^{g,\xi}$-submartingale systems}\label{decomp}

In this section, we introduce the notion of $\mathscr{Y}^{g,\xi}$-submartingale systems, and provide a $\mathscr{Y}^{g,\xi}$-Mertens-Doob-Meyer decomposition of such systems, which is a new result in the literature, even in the case $g \equiv 0$.

\subsection{Definitions and assumptions}

We give here some important definitions and assumptions. Consider a map $g$ (also called \textit{driver} in the sequel) which satisfies the following assumption:

\begin{Assumption}\label{driver}

$g$ is a measurable map from $\Omega\times [0,T] \times \mathbf{R}\times\mathbf{R}^d$ to $\mathbf{R}$ and $g(.,y,z)$ is $\FB$-predictable, for each $(y,z)\in\mathbf{R}\times\mathbf{R}^d$. There exists a constant $K_g>0$ and a random variable $\chi_g\in\mathbf{L}_2(\mathbf{R}^+)$, such that, $\forall (t,y_i,z_i)\in [0,T] \times \mathbf{R}\times\mathbf{R}^d,i=1,2$,
\begin{eqnarray*}
|g(t,0,0)|&\leq& \chi_g \quad \PB-a.s.,\\
|g(t,y_1,z_1)- g(t,y_2,z_2)| &\leq& K_g(|y_1-y_2|+|z_1-z_2|) \quad \PB-a.s.
\end{eqnarray*}
\end{Assumption}
We recall first the definition of the conditional $g$-expectation.
\begin{definition}[Conditional $g$-expectation] We recall that if $g$ {\color{black}satisfies Assumption \ref{driver}} and if $\xi$ is a square-integrable $\mathscr{F}_T$-measurable random variable, then there exists a unique solution $(X,\pi) \in \mathbf{S}_2 \times \mathbf{H}_2$ to the following BSDE
\begin{align*}
X_t=\xi+\int_t^Tg(s,X_s,\pi_s)ds-\int_t^T\pi_sdW_s \text{ for all } t \in [0,T] {\rm \,\, a.s. }
\end{align*}
For $t\in [0,T]$, the nonlinear operator $\mathcal{E}^{g}_{t,T} : \mathbf{L}_2(\mathscr{F}_T) \mapsto \mathbf{L}_2(\mathscr{F}_t)$ which maps a given terminal {\color{black}condition} $\xi \in \mathbf{L}_2(\mathscr{F}_T)$ to the first component at time $t$ of the solution of the above BSDE (denoted by $X_t$) is called conditional $g$-expectation at time $t$. It is also well-{\color{black}known} that this notion can be extended to the case where the (deterministic) terminal time $T$ is replaced by a general stopping time $\tau \in \mathcal{T}_0$ and $t$ is replaced by a stopping time $S$ such that $S \leq \tau$ a.s.

\end{definition}

Let $\xi$ be a right-continuous left-limited (RCLL) process in $\textbf{S}_2$ and $g$ {\color{black}satisfying Assumption \ref{driver}}. We denote by $\mathscr{Y}^{g,\xi}$ the nonlinear operator (semigroup) associated with the reflected BSDE with driver $g$ and lower obstacle $\xi_{\cdot}$, which is the analagous of the operator $\mathscr{E}^g$, induced by the non-reflected BSDE with driver $g$.

\begin{definition}[Nonlinear operator $\mathscr{Y}^{g,\xi}$] For each $\tau \in \mathcal{T}_0$ and each $\zeta \in \textbf{L}_2(\mathscr{F}_\tau)$ such that $\zeta \geq \xi_\tau$ a.s., we define $\mathscr{Y}^{g,\xi}_{\cdot, \tau}(\zeta):=Y_\cdot$, where $Y_\cdot$ corresponds to the first component of the solution of the reflected BSDE associated with terminal time $\tau$, driver $g$ and lower obstacle $\xi_t \textbf{1}_{\tau<t}+\zeta\textbf{1}_{\tau \geq t}.$
\end{definition}
Note that, by the flow property for reflected BSDEs, for each driver $g$, the operator $\mathscr{Y}^{g,\xi}$ is consistent
(or, equivalently, satisfies a semigroup property) with respect to terminal
condition $\zeta$. By the comparison theorem for reflected BSDEs with
RCLL obstacle, we obtain that $\mathscr{Y}^{g,\xi}$ is monotonous with respect to the
terminal condition.\\

Let us recall the definition of a $\mathcal{T}_0$-admissible system.

%
%

\begin{definition}
A family  $X=\{X(\tau), \tau \in \mathcal{T}_0\}$ is a $\mathcal{T}_0$-system (or \rm{admissible}) if for all $\tau,\tau' \in \mathcal{T}_0$,

\begin{equation}
\begin{cases}

X(\tau) \in \textbf{L}_0(\mathscr{F}_\tau),\\

X(\tau)=X(\tau') \text{  a.s. on } \{\tau=\tau'\}.
\end{cases}
\end{equation}
\end{definition}

\noindent We now introduce the notion of $\mathscr{Y}^{g,\xi}$-submartingale system (resp.   $\mathscr{Y}^{g,\xi}$-martingale system).

\begin{definition}[$\mathscr{Y}^{g,\xi}$-\textit{submartingale} system]\label{Y-submart}

An admissible family $(X(\tau), \tau \in \mathcal{T}_0)$ is said to be a $\mathscr{Y}^{g,\xi}$-{\rm submartingale {\color{black}system}} (resp. a $\mathscr{Y}^{g,\xi}$-{\rm martingale {\color{black}system}}) if $X(\tau) \geq \xi_\tau$ a.s for all {\color{black}$\tau \in \mathcal{T}_0$, $\mathbb{E}[ \underset{\tau \in \mathcal{T}_0}{\esssup}X^2(\tau)]<\infty$} and for all $\tau, \sigma \in \mathcal{T}_0$ such that $\sigma\in \mathcal{T}_\tau$ a.s.,
$$X(\tau)  \leq \mathscr{Y}^{g,\xi}_{\tau,\sigma}(X(\sigma))\,\, {\rm a.s.}\,\,\, {\rm (resp. }  X(\tau) = \mathscr{Y}^{g,\xi}_{\tau,\sigma}(X(\sigma))\,\rm a.s.) {\,\,}$$

\end{definition}

We now give the definition of a $\mathscr{Y}^{g,\xi}$-submartingale process (resp.   $\mathscr{Y}^{g,\xi}$-martingale process). \begin{definition}[\textit{Strong} $\mathscr{Y}^{g,\xi}$-\textit{submartingale} process]
An optional process $(Y_t) \in \mathbf{S}_2$ satisfying $Y_\sigma \geq \xi_\sigma$ a.s. for all $\sigma \in \mathcal{T}_0$ and such that $\mathbb{E}[\underset{\tau \in \mathcal{T}_0}{\esssup} (Y_\tau)^2]<\infty$  is said to be a strong $\mathscr{Y}^{g,\xi}$-submartingale {\color{black} (resp. a  strong $\mathscr{Y}^{g,\xi}$-{\rm martingale})} if 
$Y_S \leq \mathscr{Y}^{g,\xi}_{\, S,\tau}(Y_\tau)$ {\color{black} (resp $Y_S=\mathscr{Y}^{g,\xi}_{\, S,\tau}(Y_\tau)$)} a.s. on $S \leq \tau$, for all $S,\tau \in \mathcal{T}_0$.
\end{definition}

\subsection{$\mathscr{Y}^{g,\xi}$-submartingale systems: aggregation and Mertens-Doob-Meyer decomposition}
In this section, we prove the $\mathscr{Y}^{g,\xi}$-Mertens-Doob-Meyer decomposition of $\mathscr{Y}^{g,\xi}$-submartingale systems. To this purpose, using tools from the general theory of processes, we first provide an aggregation result for $\mathscr{Y}^{g,\xi}$-submartingale systems. 
\begin{theorem}[Aggregation of a $\mathscr{Y}^{g,\xi}$-submartingale {\color{black}system} by a right-l.s.c.process]\label{agreg}
Let $(X(S), S \in \mathcal{T}_0)$ be an $\mathscr{Y}^{g,\xi}$-submartingale {\color{black}system}. Then there exists a right-lower semicontinuous optional process $(X_t)$ belonging to $\mathbf{S}_2$ which aggregates the family $(X(S),\,\, S \in \mathcal{T}_0)$, that is such that $X(S)=X_S$ a.s. for all $S \in \mathcal{T}_0$. Moreover, the process $(X_t)$ is a strong $\mathscr{Y}^{g,\xi}$-submartingale, that is, for each $S \in \mathcal{T}_0$, $X_S \in L^2$, $X_S \geq \xi_S$ a.s. and for all $S, S' \in \mathcal{T}_0$ such that $S \geq S'$ a.s., $\mathscr{Y}^{g,{\color{black}\xi}}_{S',S}(X_S) \geq X_{S'}$ a.s.
\end{theorem}
\dproof
Fix $\tau \in \mathcal{T}_0$ and let $(\tau_n)_{n \in \mathbb{N}}$ be a non-increasing sequence of stopping times in $\mathcal{T}_\tau$ such that $\tau_n \downarrow \tau$ a.s. and for all $n \in \mathbb{N}$, we have $\tau_n>\tau$ a.s. on $\{  \tau<T\}$ and such that $\underset{n \rightarrow \infty} {\lim}X(\tau_n)$ exists a.s. By {\color{black}Definition \ref{Y-submart}} of the operator $\mathscr{Y}^{g,\xi}$, we derive that
\begin{align}\label{ineq}
    X(\tau) \leq \mathscr{Y}^{g,\xi}_{\tau, \tau_n}(X(\tau_n))\,\ \text{ a.s. for all } n \in \mathbb{N}.
\end{align}
Since the sequence of stopping times $(\tau_n)_n$ is non-increasing and the operator $\mathscr{Y}^{g,\xi}$ is consistent, we derive that
\begin{align}
    \mathscr{Y}^{g,\xi}_{\tau, \tau_n}(X(\tau_n))=\mathscr{Y}^{g,\xi}_{\tau, \tau_{n+1}}(\mathscr{Y}^{g,\xi}_{\tau_{n+1}, \tau_{n}}(X(\tau_{n}))) \geq \mathscr{Y}^{g,\xi}_{\tau, \tau_{n+1}}(X(\tau_{n+1})) \text{\,\,a.s.,}
\end{align}
where the last inequality follows by \eqref{ineq} {\color{black} and the comparaison theorem for reflected BSDEs}. This implies that the sequence $\mathscr{Y}^g_{\tau, \tau_n}(X(\tau_n))_{n \in \mathbb{N}}$ is non-increasing, and thus it converges almost surely. Moreover,
\begin{align}
    X(\tau) \leq \underset{n \rightarrow \infty}{\lim} \downarrow \mathscr{Y}^{g,\xi}_{\tau, \tau_n}(X(\tau_n)) \text{\,\,a.s.}
\end{align}
Since, by the right-continuity of the obstacle, we have $\underset{n \rightarrow \infty}{\lim }X(\tau_n) \geq \xi_\tau$ a.s., we can use the continuity result of Reflected BSDEs with respect to terminal time and terminal condition (see Proposition 3.13 in \cite{DQS16}) and obtain
\begin{align}
    X(\tau) \leq \underset{ n \rightarrow \infty}{\lim} \mathscr{Y}^{g,\xi}_{\tau, \tau_n}(X(\tau_n)) \leq \mathscr{Y}^{g,\xi}_{\tau, \tau}(\underset{ n \rightarrow \infty}{\lim} X(\tau_n)) = \underset{ n \rightarrow \infty}{\lim} X(\tau_n) \,\, \text{a.s.}
\end{align}
By Lemma 5 in \cite{DL82}, we conclude that the family $(X(S), \,\, S \in \mathcal{T}_0)$ is right lower semicontinuous. It follows from Theorem 4 in \cite{DM80} that there exists a right lower-semicontinuous optional process $(X_t)$ which aggregates the family $(X(S), \,\, S \in \mathcal{T}_0)$, which is {\color{black}a strong} $\mathscr{Y}^{g,\xi}$-submartingale.
\fproof

\begin{Remark}\label{rlsc}
By the above Theorem, we get that strong $\mathscr{Y}^{g,\xi}$-submartingale processes are right-lower semicontinuous (r.l.s.c. for short).
\end{Remark}

We now show the nonlinear $\mathscr{Y}^{g,\xi}$-Mertens decomposition of $\mathscr{Y}^{\,g,\xi}$-submartingales, which represents, to the best of our knowledge, a new result in the literature. Our proof is simple, avoiding the standard penalization technique used in the setting of standard $\mathcal{E}^g$-{\color{black}submartingales}.\\

\noindent Before giving the main theorem, we recall the following definition of mutually singular measures associated with RCLL predictable processes.

\begin{definition}
Let $A=(A_t)_{0 \leq t \leq T}$ and $A'=(A'_t)_{0 \leq t \leq T}$ belonging to ${\color{black}{\mathbf K}_2}$. The measures $dA_t$ and $dA'_t$ are said to be mutually singular and we write $dA_t \perp dA'_t$ if there exists $D \in \mathcal{P}$ such that 
$$\mathbb{E}\left[\int_0^T 1_{D^c}dA_t\right]=\mathbb{E}\left[\int_0^T 1_{D}dA'_t\right]=0.$$
\end{definition}

We now prove the following decomposition \textit{\`a la Mertens.}

\begin{theorem}[$\mathscr{Y}^{g,\xi}$-Mertens decomposition of $\mathscr{Y}^{g,\xi}$-submartingales]\label{eee1}

Let $(Y_t)$ be an optional process such that $\mathbb{E}[\underset{\tau \in \mathcal{T}_0}{\esssup} (Y_\tau)^2]<\infty$  and $(\xi_t)$ be a right-continuous left-limited strong semimartingale  such that $\mathbb{E}[\underset{\tau \in \mathcal{T}_0}{\esssup} (\xi_\tau)^2]<\infty$. The process $(Y_t)$ is a strong $\mathscr{Y}^{\,g,\xi}$-submartingale process if and only if there exist two non-decreasing right-continuous predictable processes $A,K \in \mathbf{K}_2$ such that 
$A_0=0$ and $K_0=0$, a non-decreasing right-continuous adapted purely discontinuous process $C'$ in $\mathbf{S}_2$ with $C'_{0^-}=0$ and a process $Z \in \mathbf{H}_2$ such that a.s. for all $t \in [0,T]$,
\begin{align}
\begin{cases}
&Y_t=Y_T+\displaystyle\int_t^Tg(s,Y_s,Z_s)ds-\int_t^T Z_s dW_s+A_T-A_t-K_T+K_t-C'_{T^-}+C'_{t^-},\,\, \label{i30}  \\
&{\color{black}Y_\tau \geq \xi_\tau }; {\rm\,\, a.s.\,\, for \,\, all\,\, } \tau \in \mathcal{T}_0;\\
& \displaystyle\int_0^T(Y_{s^-}-\xi_{s^-})dA_s=0 {\rm \,\, a.s.} ; \\
&dA_t \perp dK_t.
\end{cases}
\end{align}
Moreover, this decomposition is unique.
\end{theorem}

\begin{proof}
Fix $S \in \mathcal{T}_0$. Since $(Y_t)$ is a strong $\mathscr{Y}^{\,g,\xi}$-submartingale, we derive that for each $\tau \in \mathcal{T}_S$, we have $Y_S \leq \mathscr{Y}^{\,g,\xi}_{\, S,\tau}(Y_\tau)$ a.s. By the characterization of the solution of a reflected BSDE in terms of an optimal stopping problem with $g$-expectations,  we have $Y_S \leq \underset{S' \in \mathcal{T}_S}{\esssup}\, \mathcal{E}^{g}_{S, S' \wedge \tau}(Y_\tau \textbf{1}_{S' \geq \tau}+ \xi_{S'}\textbf{1}_{S' < \tau}).$ By arbitrariness of $\tau \in \mathcal{T}_S$, hence we get
\begin{align}\label{i20}
Y_S \leq \essinf_{\tau \in \mathcal{T}_S} \esssup_{S' \in \mathcal{T}_S} \mathcal{E}^{g}_{S, S' \wedge \tau}(Y_\tau \textbf{1}_{S' \geq \tau}+ \xi_{S'}\textbf{1}_{S' < \tau})\,\,\, {\rm a.s. }
\end{align}
Now, one can remark that we have
\begin{align*}
Y_S=\esssup_{S' \in \mathcal{T}_S} \mathcal{E}^{g}_{S,S \wedge S'}\left(Y_S \textbf{1}_{S' \geq S}+\xi_{S'} \textbf{1}_{S>S'}\right) {\rm\,\, a.s.}
\end{align*}

As $S \in \mathcal{T}_S$, we deduce:
\begin{align}\label{i21}
Y_S \geq \essinf_{\tau \in \mathcal{T}_S} \esssup_{S' \in \mathcal{T}_S} \mathcal{E}^{g}_{S,\tau \wedge S'}\left(Y_\tau \textbf{1}_{S' \geq \tau}+\xi_{S'} \textbf{1}_{\tau>S'}\right) {\rm\,\, a.s.}
\end{align}

The  inequalities \eqref{i20} and \eqref{i21} allow to conclude that
\begin{eqnarray}\label{DRBSDE_Dynkin} Y_S =  \essinf_{\tau \in \mathcal{T}_S} \esssup_{S' \in \mathcal{T}_S} \mathcal{E}^{g}_{S,S \wedge S'}\left(Y_\tau \textbf{1}_{S' \geq \tau}+\xi_{S'} \textbf{1}_{\tau>S'}\right) {\rm\,\, a.s.}\end{eqnarray}

We now use the characterization of the value function of  generalized Dynkin game in terms of solution to doubly reflected BSDE, first introduced by Cvitanic et al. \cite{CK96}. Namely, Theorem 4.5 in \cite{GIOOQ17} ensures that, given two obstacles $(\xi_t)$ and $(\zeta_t)$ supposed to be resp. r.u.s.c. and r.l.s.c., and satisfying the Mokobodzki's condition, the value function of a Generalized Dynkin game given by $$\bar{Y}_S:=\underset{\tau \in \mathcal{T}_S} {\essinf}\underset{\sigma \in \mathcal{T}_S}{\esssup}\,\mathcal{E}^{g}_{S, \tau \wedge \sigma}[\xi_{\tau} \textbf{1}_{\tau <\sigma}+\zeta_{\sigma} \textbf{1}_{\sigma\leq \tau}]=\underset{\sigma \in \mathcal{T}_S}{\esssup}\underset{\tau \in \mathcal{T}_S} {\essinf}\,\mathcal{E}^{g}_{S, \tau \wedge \sigma}[\xi_{\tau} \textbf{1}_{\tau <\sigma}+\zeta_{\sigma} \textbf{1}_{\sigma\leq \tau}]\,,$$ is such that $\bar{Y}$ is the first component of the solution of the doubly reflected BSDE with driver $g$ and obstacles $(\xi_t)$ and $(\zeta_t)$. 
Hence, this result can be applied in our setting, since $(\xi_t)$ is RCLL and $(Y_t)$ is r.l.s.c. (see Remark \ref{rlsc}) and the Mokobodzki's condition is satisfied. Thus,  \reff{DRBSDE_Dynkin} implies that the process $(Y_t)$ coincides with the first component of the solution of the doubly reflected BSDE associated with obstacles $(Y_t)$ and  $(\xi_t)$, i.e. there exist two non-decreasing right-continuous predictable processes $A,K \in \mathbf{K}_2$ such that 
$A_0=0$ and $K_0=0$, two non-decreasing right-continuous adapted purely discontinuous processes $C,C'$ in $\mathbf{S}_2$ with $C_{0^-}=0$ and $C'_{0^-}=0$ and a process $Z \in \mathbf{H}_2$ such that a.s. for all $t \in [0,T]$,
\begin{align}
\begin{cases}
&Y_t=Y_T+\displaystyle\int_t^Tg(s,Y_s,Z_s)ds-\int_t^T Z_s dW_s+A_T-A_t+C_{T^-}-C_{t^-}-K_T+K_t-C'_{T^-}+C'_{t^-},\,\, \label{i31}  \\
&Y_t \geq \xi_t\,{\rm \,\, a.s.\,\,}; \\
&\displaystyle \int_0^T(Y_{s^-}-\xi_{s^-})dA_s=0 {\rm \,\, a.s.} ; \,\, (Y_\tau-\xi_\tau)(C_\tau-C_{\tau^-})=0\, {\rm\,\, a.s.\,\, for \,\, all\,\, } \tau \in \mathcal{T}_0; \\
&dA_t \perp dK_t\,; \,\, \, dC_t \perp dC'_t.
\end{cases}
\end{align}
Using the above equation, we can observe that $\Delta C_t=(Y_{t^+}-Y_t)^-$ and $\Delta C'_t=(Y_{t^+}-Y_t)^+$ for all $t$ a.s., where $Y_{t^+}$ stands for the process of right-limits, which exist by the above decomposition. Since the process $(Y_t)$ is r.l.s.c. by Remark \ref{rlsc}, we get that $\Delta C_t=0$ for all $0 \leq t \leq T$ a.s. {\color{black}Furthermore, the condition $Y_\tau \geq \xi_\tau$ {\rm\, a.s.\,for \, all\, } $\tau \in \mathcal{T}_0$ is satisfied due to the definition of a strong $\mathscr{Y}^{\,g,\xi}$-submartingale process.} The result follows.

Let us now show the converse implication. The reflected BSDE $\eqref{i30}$ can be seen as  a reflected BSDE associated to the \textit{generalized driver} $g(t,\omega,y,z)dt-dK_t-dC'_{t^-}$.

Fix $\tau \in \mathcal{T}_S$. Using the flow property for reflected BSDEs and their representation as the value function of an optimal stopping problem, we get
\begin{align}
Y_S=\esssup_{S' \in \mathcal{T}_S} \mathcal{E}^{g-dK-dC'}_{S,S' \wedge \tau} \left[ Y_\tau \textbf{1}_{\tau \leq S'}+\xi_{S'} \textbf{1}_{S'<\tau}\right] {\rm \,\, a.s.}
\end{align}

Using the comparison theorem for BSDEs with \textit{generalized driver}, we deduce that
\begin{align}
Y_S \leq \esssup_{S' \in \mathcal{T}_S} \mathcal{E}^{g}_{S,S' \wedge \tau} \left[ Y_\tau \textbf{1}_{\tau \leq S'}+\xi_{S'} \textbf{1}_{S'<\tau}\right] {\rm \,\, a.s.},
\end{align}
which implies that
$$Y_S \leq \mathscr{Y}^{\,g,\xi}_{S,\tau}[Y_\tau] {\rm \,\, a.s., \,\, for \,\, all\,\,} \tau \in \mathcal{T}_S.$$
The uniqueness of the decomposition follows by the uniqueness of the decomposition of a semimartingale and the fact that the measures $dA$ and $dK$ (resp. $dC$ and $dC'$) are mutually singular.
\fproof
\begin{Remark}
From the previous Theorem, we deduce that strong $\mathscr{Y}^{g,\xi}$-submartingales have left and right limits.
\end{Remark}
\end{proof}

In the case when the process $\mathscr{Y}^{g,\xi}$ is right-continuous left-limited (RCLL), we obtain the following $\mathscr{Y}^{g,\xi}$-\textit{Doob-Meyer decomposition}, which also seems a new result in the literature.

\begin{theorem}[$\mathscr{Y}^{g,\xi}$-Doob-Meyer decomposition of RCLL $\mathscr{Y}^{g,\xi}$-submartingales]\label{DMRCLL}
Let $(Y_t)$ be an optional process such that $\mathbb{E}[\underset{\tau \in \mathcal{T}_0}{\esssup} (Y_\tau)^2]<\infty$  and $(\xi_t)$ be a right-continuous left-limited strong semimartingale such that $\mathbb{E}[\underset{\tau \in \mathcal{T}_0}{\esssup} (\xi_\tau)^2]<\infty$. The process $(Y_t)$ is a \textit{RCLL strong $\mathscr{Y}^{\,g,\xi}$-submartingale process} if and only if there exist two non-decreasing right-continuous predictable processes $A,K \in \mathbf{K}_2$ such that 
$A_0=0$ and $K_0=0$ and a process $Z \in \mathbf{H}_2$ such that a.s. for all $t \in [0,T]$,
\begin{align}\label{i30}
\begin{cases}
&Y_t=Y_T+\int_t^Tg(s,Y_s,Z_s)ds-\int_t^T Z_s dW_s+A_T-A_t-K_T+K_t;\,\,  \\
&Y_t \geq \xi_t \,{\rm \,\, a.s.\,\,}; \nonumber \\
& \int_0^T(Y_{s^-}-\xi_{s^-})dA_s=0 {\rm \,\, a.s.}; \nonumber \\
&dA_t \perp dK_t. \nonumber
\end{cases}
\end{align}
Moreover, this decomposition is unique.
\end{theorem}

\begin{proof}
Let $(Y_t)$ be a RCLL $\mathscr{Y}^{g,\xi}$-submartingale. We can thus apply Theorem \ref{eee1} and obtain the existence of the processes $(Z, A, K, C')$ $\in \mathbf{H}_2 \times (\mathbf{K}_2)^3$ such that \eqref{i30} holds. Due to this equation, we have $\Delta C^{'}_t=(Y_{t^+}-Y_t)$. Since the process $(Y_t)$ is right-continuous, then the process $C'=0$. 
The result follows.
\fproof
\end{proof}

\section{Application to BSDEs with \textit{weak constraints} at stopping times}\label{BSDE}
In this section, we introduce a new class of BSDEs with \textit{weak constraints} at stopping times, which are related to the partial hedging of American options. Using the results from the first part, we will provide a representation of the family of $\tau$-minimal values in terms of a \textit{reflected BSDE}.
\label{Definition}



\subsection{Definition and Assumptions}

We introduce here the new mathematical object.

\begin{definition}(BSDEs with weak constraints at stopping times)
Given a measurable map $\Psi: [0,T] \times  {\color{black}\Omega\times\mathbf{R}} \rightarrow [0,1]$, with , $\tau \in \mathcal{T}_0$ and $\mu \in \mathbf{L}_0([0,1], \mathscr{F}_\tau)$, we say that $(Y,Z) \in \mathbf{S}_2 \times \mathbf{H}_2$ is a supersolution of the BSDE with generator $g:\Omega \times [0,T] \times \mathbf{R} \times \mathbf{R} ^d \rightarrow \mathbf{R}$ and  weak constraints  $({\color{black}\Psi},\mu,\tau)$   if, for any $0\leq t\leq s\leq T$,

\begin{eqnarray}
Y_t \geq Y_s&+\displaystyle\int_t^sg(u,Y_u,Z_u)du-\displaystyle\int_t^sZ_udW_u;\label{Dynamics} \\
& \mathbb{E}_\tau[\Psi(\theta, Y_\theta)] \geq \mu \,\,\, \text{for all  } \theta \in \mathcal{T}_\tau. \label{weak_cond}
\end{eqnarray}

\end{definition}

\noindent The terminology \textit{BSDEs with weak constraints at stopping times} is due to the fact that, given a stopping time $\tau \in \mathcal{T}_0$ and a threshold $\mu \in \textbf{L}_0$, the first component of the solution of the above BSDE, here denoted by $(Y_t)$, satisfies the condition $\mathbb{E}_\tau[\Psi(\theta, Y_\theta)] \geq \mu$  for all  $\theta \in \mathcal{T}_\tau$. The wellposedness of this class of BSDEs, under appropriate assumptions, is discussed below.\\
{\color{black}\begin{Remark}
The set $[0, 1]$ in the above definition is chosen for ease of notation and can be replaced by any closed interval $ U \subset \mathbf{R} \cup \{-\infty\}$.
\end{Remark}}

{\color{black}Throughout the paper, }the driver $g$ is supposed to satisfy Assumption \ref{driver} and the map $\Psi$ the following assumption.

\begin{Assumption}\label{AssPsi}
For $dt \times d\mathbb{P}$-a.e. $(t, \omega) \in [0,T] \times \Omega$, the map $y \in \mathbf{R} \mapsto \Psi(t,\omega,y)$ is non-decreasing, right-continuous, valued in $[0,1] \cup \{-\infty\}$ and its left-continuous inverse $\Phi(t,\omega,\cdot)$ satisfies $\Phi:[0,T] \times \Omega \times [0,1] \mapsto [0,1]{\color{black}\cup \{+\infty\}}$ is measurable.
\end{Assumption}
By left-continuous inverse, we mean the left-continuous map defined, for any fixed $(t,\omega,x)$, by
$$\Phi(t,\omega,x):= \inf\{y \in \mathbf{R}, \,\, \Psi(t,\omega,y)\geq x\}\;,$$
and hereby satisfies 
$$\Phi (t,\omega, \Psi(t,\omega,x)) \leq x \leq \Psi (t,\omega, \Phi(t,\omega,x)).$$\\

\noindent \textit{Well-posedness of BSDEs with weak constraints at stopping times.}\\
We address here the question of the well-posedness of the BSDE with weak constraints. Let $\xi$ be a square integrable $\mathscr{F}_T$-measurable random variable {\color{black}valued on $[0,1]$} such that $\mathbb{E}_\tau[\xi]=\mu$ a.s. Due to the martingale representation theorem, there exists $\alpha \in \mathbf{H}_2$ such that $$M_T^\alpha=\xi \quad\text{ a.s., where\,}\quad M_T^\alpha=\mu+\int_\tau^T \alpha_s dW_s .$$
The solution $(Y^\alpha_t,Z^\alpha_t)$ of the reflected BSDE associated with the driver $g$ and obstacle $(\Phi(t,M_t^\alpha))$ (which exists under the above assumptions) is a supersolution of the BSDE with weak constraints. {\color{black} Indeed, the inequality \eqref{Dynamics} is satisfied due to the increasing property of the finite-variation process appearing in the decomposition of the RBSDE $(Y^\alpha_t,Z^\alpha_t)$. Furthermore, we obtain the inequality \eqref{weak_cond} by using the martingale property of $M^\alpha$ and applying the map $\Phi$ on both sides of the constraint \eqref{weak_cond}.}

Due to the weak constraints, we do not expect uniqueness of the solution.\\

\noindent We now introduce the set $\Theta(\tau,\mu)$ of $(\tau,\mu)$-initial supersolutions, which is defined as follows:
$$\Theta(\tau,\mu):= \{Y_\tau:\,\, (Y,Z) {\rm  \,\, is\,\, a\,\, supersolution\,\, of\, }  \eqref{Dynamics} {\rm \,\, and \,\,} \eqref{weak_cond}\}.$$

\subsection{Characterization of the minimal $\tau$-values family as a control-stopping game problem}

\noindent The aim of this subsection is to study the lower bound of the set $\Theta(\tau,\mu)$, that is $\essinf \Theta(\tau,\mu)$, which corresponds to the efficient price of an American option under the risk constraint $\mathbb{E}_\tau[\Psi(\theta, Y_\theta)] \geq \mu$ a.s., for all $\theta \in \mathcal{T}_\tau$. In particular, we will provide a characterization in terms of a control-stopping game problem.

\noindent The first step is to reformulate the problem of interest into an equivalent one involving "strong" constraints, in a similar manner as for the efficient hedging of European options. We refer to Bouchard, Elie, Touzi \cite{BET10} in the Markovian framework or Bouchard, Elie, Reveillac \cite{BER15} in the non-Markovian setting. \\
For this purpose, let $\mathscr{A}_{\tau, \mu}$ denote the set of elements $\alpha \in \mathbf{H}_2$ such that 
\begin{equation}
M^{^{\tau,\mu,\alpha}}_{\cdot}:= \mu+\int_\tau^{\tau\vee.}\alpha_s dW_s\quad \text{takes values in} \,[0, 1].
\end{equation}
The main difficulty in our setting case is due to the fact that we can {\it a priori} only obtain an equivalent formulation in which the controlled martingale depends on the chosen stopping time $\theta$, i.e. for each $\theta \in \mathcal{T}_\tau$ there exists $\alpha^\theta \in \mathscr{A}_{\mu,\tau}$ such that $\mathbb{E}_\tau[\Psi(\theta,Y_\theta)] \geq \mu$ is equivalent to $Y_\theta \geq \Phi(\theta,M_\theta^{\tau,\mu,\alpha^{\theta}})$ a.s.\\

We show in the next Lemma that we can overcome this issue and obtain the existence of a control process independent of the stopping time.
\begin{lemma}[\textit{Equivalent reformulation with strong constraints}]\label{keyres}
Fix $\tau \in \mathcal{T}_0$ and $\mu \in \mathbf{L}_0([0,1], \mathscr{F}_\tau)$.  Let $(Y_t, Z_t) \in \mathbf{S}_2 \times \mathbf{H}_2$ be a supersolution of the BSDE \eqref{Dynamics}-\eqref{weak_cond}. Then the \textit{weak constraint} $\mathbb{E}\left[\Psi(\theta, Y_\theta)| \mathscr{F}_\tau  \right] \geq \mu$ for all $\theta \in \mathcal{T}_{\tau}$ is equivalent to the \textit{strong constraint} 
$Y_\theta \geq \Phi(\theta,M_\theta^{^{\tau,\mu,\alpha}})$ a.s., for a given $\alpha \in \mathscr{A}_{\tau, \mu}$ and for all $\theta \in \mathcal{T}_{\tau}$.
\end{lemma}

\begin{proof}
We first show that \eqref{weak_cond} implies that there exists $\alpha \in \mathscr{A}_{\tau, \mu}$ such that $Y_\theta \geq \Phi(\theta,M_\theta^{^{\tau,\mu,\alpha}})$ a.s. for all $\theta \in \mathcal{T}_{\tau}$. To this purpose, we define, for each $\sigma \in \mathcal{T}_0,$ the $\mathscr{F}_\sigma$-measurable random variable
\begin{align}\label{value}
V(\sigma):=  \essinf_{\tau \in \mathcal{T}_\sigma} \mathbb{E}\left[\Psi(\tau, Y_\tau)| \mathscr{F}_\sigma\right].
\end{align}
By classical results of the general theory of processes (see \cite{DM80}) , the family $(V(\sigma),\,\,\sigma \in \mathcal{T}_0)$ is a submartingale {\color{black}system}, which can be aggregated by an optional process $(V_t)$ admitting the Mertens decomposition:
\begin{align*}
V_t:=N_t+A_t+C_{t^-},
\end{align*}
where $N$ is a square integrable martingale, $A$ is a non-decreasing RCLL predictable process such that $A_0=0$ and $C$ is a non-decreasing right-continuous adapted process, purely discontinuous satisfying $C_{0^-}=0$. By the martingale representation theorem, there exists $\alpha \in \mathbf{H}_2$ such that $N_\theta=N_\tau+\displaystyle\int_\tau^\theta \alpha_sds$, for all $\theta \in \mathcal{T}_\tau$.\\

\noindent Since for all $\theta \in \mathcal{T}_\tau$ we have $\mathbb{E}\left[\Psi(\theta,Y_\theta)| \mathscr{F}_\tau\right] \geq \mu $ a.s., we get that $\underset{\theta \in \mathcal{T}_\tau}{\essinf}\,\mathbb{E}\left[\Psi(\theta,Y_\theta)| \mathscr{F}_\tau\right] \geq \mu$ a.s. Hence, by using the definition of $V$ (see \eqref{value}), we obtain 
\begin{align}\label{tau}
V_\tau=N_\tau+A_\tau+C_{\tau^-} \geq \mu {\rm\,\, a.s.}
\end{align}
We now fix $\theta \geq \tau$. We have $\Psi(\theta, Y_\theta)=\mathbb{E}\left[\Psi(\theta,Y_\theta)| \mathscr{F}_\theta\right] \geq\underset{\sigma  \in \mathcal{T}_\theta}{\essinf}\,\mathbb{E}\left[\Psi(\sigma,Y_\sigma)| \mathscr{F}_\theta \right]=V_\theta {\rm \,\, a.s.}$
This observation, together with \eqref{tau}, implies
$$\Psi(\theta,Y_\theta) \geq N_\theta+A_\theta+C_{\theta^-}=N_\tau+A_\tau+C_{\tau^-}+\int_\tau^\theta \alpha_sdW_s+A_\theta-A_\tau+C_{\theta^-}-C_{\tau^-} {\color{black}\geq \mu+\int_\tau^\theta \alpha_sdW_s},$$
{\color{black}where the last inequality follows from inequality \eqref{tau} and the fact the processes $A$ and $C$ are non-decreasing.

Now, define the stopping time $\theta^{\alpha}:=\inf \{ s \geq \tau, \,\, M_s^{\tau,\mu, \alpha}=0\}$ and the control process $\bar{\alpha}:=\alpha\textbf{1}_{[0,\theta^{\alpha}]},$ which clearly belongs to $\mathscr{A}_{\tau,\mu}$.
Using the above inequality and the fact that for all $\theta \in \mathcal{T}_\tau$, $\Psi(\theta,Y_\theta) \geq 0$ a.s. (by Assumption \ref{AssPsi}), we finally derive 
$$ Y_\theta \geq \Phi(\theta, M_\theta^{^{\tau,\mu,\bar{\alpha}}})  {\rm \,\, a.s.}$$}
The second implication is trivial.

\fproof\\
\end{proof}
From the previous Lemma, we easily deduce the following result.

\begin{proposition}\label{res}
Fix $\tau \in \mathcal{T}_0, \,\, \mu \in \mathbf{L}_{0}([0,1], \mathscr{F}_\tau)$. Then $(Y,Z) \in \mathbf{S}_2 \times \mathbf{H}_2$ is a supersolution of the BSDE \eqref{Dynamics}-\eqref{weak_cond} if and only if $(Y,Z)$ satisfies \eqref{Dynamics} and there exists $\alpha \in \mathscr{A}_{\tau, \mu}$ such that $Y_\nu \geq \underset{\theta \in \mathcal{T}_{\nu}}{\esssup}\,\mathcal{E}^{g}_{\nu,\theta}[\Phi(\theta,M_\theta^{\tau,\mu,\alpha})]$ a.s. for all $\nu \in \mathcal{T}_\tau.$
\end{proposition}

\begin{proof}
Let $(Y,Z)$ be a supersolution of BSDE \eqref{Dynamics}-\eqref{weak_cond}. By Lemma \ref{keyres}, there exists $\bar{\alpha} \in \mathscr{A}_{\tau,\mu}$ such that for all $\theta \in \mathcal{T}_\tau$ we have $\Psi(\theta, Y_\theta)\geq M_\theta^{\tau,\mu, \bar{\alpha}}.$ a.s. The monotonocity of the map $\Phi$ and the above inequality imply that:
\begin{align*}
Y_{\theta} \geq \Phi(\theta, M_\theta^{\tau,\mu, \bar{\alpha}}) {\rm \,\, a.s.}
\end{align*}
By the comparison theorem for BSDEs, we get that for all $\theta \in \mathcal{T}_\nu$, we have $Y_\nu \geq \mathcal{E}^{g}_{\nu, \theta}[\Phi(\theta, M_\theta^{\tau,\mu, \bar{\alpha}})]$ a.s. Now, by arbitrariness of $\theta \in \mathcal{T}_\nu$ we finally obtain:
\begin{align}
Y_\nu \geq \esssup_{\theta \in \mathcal{T}_\nu}\mathcal{E}^{g}_{\nu,\theta}[\Phi(\theta,M_\theta^{\tau,\mu, \bar{\alpha}})] {\rm \,\, a.s.}
\end{align}
Let us show the converse implication. Let $\alpha \in \mathscr{A}_{\tau, \mu}$ such that for all $\nu \in \mathcal{T}_\tau$, we have $Y_\nu \geq \Phi(\nu, M_\nu^{\tau, \mu, {\alpha}})$ a.s. Hence we get $\Psi(\nu, Y_\nu)\geq M_\nu^{\tau,\mu, {\alpha}}$ a.s. This implies that $(Y,Z)$ satisfies $\eqref{Dynamics}$ and $\eqref{weak_cond}$.
\fproof\\
\end{proof}
Using the above results, we show in the following proposition how to relate the lower bound of the family $\Theta(\tau,\mu)$ to the value of a stochastic control/optimal stopping game. To this aim, we define the value function
\begin{equation}\label{value_func}
\mathcal{Y}(\tau, \mu):=\underset{\alpha \in \mathscr{A}_{\tau, \mu}} {\essinf}\,\underset{\theta \in \mathcal{T}_{\tau}}{\esssup}\,\mathcal{E}^{g}_{\tau,\theta}[\Phi(\theta,M_\theta^{\tau,\mu,\alpha})].
\end{equation}
\begin{proposition}
We have $\essinf \Theta(\tau,\mu)=\mathcal{Y}(\tau,\mu)$ a.s.
\end{proposition}
\begin{proof} \quad\\
Let $Y_\tau \in \Theta(\tau,\mu)$. By Proposition \ref{res}, we obtain that there exists $\bar{\alpha} \in \mathscr{A}_{\tau,\mu}$ such that $Y_\tau \geq \underset{\theta \in \mathcal{T}_\tau}{\esssup}\,\mathcal{E}^{g}_{\tau,\theta}[\Phi(\theta,M_\theta^{\tau, \mu, \bar{\alpha}})]$ a.s., which clearly implies that $$Y_\tau \geq \underset{\alpha \in \mathscr{A}_{\tau, \mu}} {\essinf}\,\underset{\theta \in \mathcal{T}_\tau} {\esssup}\,\mathcal{E}^{g}_{\tau, \theta}[\Phi(\theta,M_\theta^{\tau, \mu, \alpha}))]=\mathcal{Y}(\tau,\mu) {\rm \,\, a.s.}$$ By arbitrariness of $Y_\tau$, we derive that $\essinf \Theta(\tau, \mu) \geq \mathcal{Y}(\tau, \mu)\,\,\,$ a.s.\\ Conversely, we have that, for each $\alpha \in \mathscr{A}_{\tau, \mu}$, $\underset{\theta \in \mathcal{T}_\tau} {\esssup}\,\mathcal{E}^{g}_{\tau, \theta}[\Phi(\theta,M_\theta^{\tau, \mu, \alpha})]$ belongs to $\Theta(\tau, \mu)$, which leads to $$\underset{\theta \in \mathcal{T}_\tau}{\esssup}\,\mathcal{E}^{g}_{\tau, \theta}[\Phi(\theta,M_\theta^{\tau, \mu, \alpha})]\geq \essinf \Theta(\tau, \mu) {\rm \,\, a.s.}$$ By taking the essential infimum on $\alpha \in \mathscr{A}_{\tau, \mu}$, the result follows.\fproof
\end{proof}

\subsection{A reflected BSDE representation of the \textit{minimal} $\tau$-values family}
\label{characterization}
\noindent In this subsection, for an initial threshold $m_0$ at time $0$ and a given admissible control $\alpha \in \mathscr{A}_{0,m_0}$, we consider the dynamic 'threshold' ($M_t^{0,m_0,\alpha}$) and provide a reflected BSDE representation of the minimal $\tau$-values family $\mathcal{Y}(\tau,M_\tau^\alpha)$.\\

\noindent For ease of notations, we  fix $m_0 \in [0,1]$ and set
\begin{align}
\begin{cases}
M_t^{\alpha}:=M_t^{0,m_0, \alpha}, \,\,\, \mathscr{A}_\tau^\alpha:=\{\alpha' \in \mathscr{A}_{\tau, M_\tau^{\alpha}}: \,\,\, \alpha'=\alpha \,\,\, dt \times d\PB \text{ on } [\![0, \tau]\!]  \}, \\
\mathscr{A}_0:=\mathscr{A}_{0,m_0} \text{ and } \mathcal{Y}^{\alpha}(\tau):=\mathcal{Y}(\tau, M_\tau^{\alpha}) \text{ for } \alpha \in \mathscr{A}_0, t\in [0,T] \text{ and } \tau \in \mathcal{T}_0.
\end{cases}
\end{align}
{\color{black}We make here the following assumption concerning the regularity of the map $\Phi$, which will be used only in this Section.}
\begin{Assumption}\label{assumpt}
Assume that the map $\Phi$ is continuous with respect to $t$ and $m$. Moreover, suppose that, for each $\alpha \in \mathscr{A}_0$, the process $\Phi(\cdot,M_\cdot^\alpha)$ is a strong semimartingale.
\end{Assumption}

Using the characterization of the first component of the solution of a nonlinear reflected BSDE as the value of an optimal stopping with nonlinear BSDEs, we obtain that, for $\tau \in \mathcal{T}_0$ and $\mu \in \textbf{L}_0([0,1], \mathscr{F}_\tau)$,
\begin{align}\label{NewName}
\mathcal{Y}(\tau,\mu)=\underset{\alpha \in \mathscr{A}_{\tau,\mu}}\essinf\; \mathscr{Y}^{g,\alpha}_{\tau,T}[\Phi(T,M_T^{\tau, \mu,\alpha})],
\end{align}
where the operator $\mathscr{Y}^{g,\alpha}$ is associated to the obstacle process $\Phi(\cdot,M_\cdot^{\tau, \mu, \alpha})$.\\

In particular, for a given $\alpha \in \mathscr{A}_0$, $\mathcal{Y}^{\alpha}$ can be rewritten as follows:
\begin{align}\label{NewName}
\mathcal{Y}^{\alpha}(\tau)=\underset{\alpha' \in \mathscr{A}^\alpha_{\tau}}\essinf\; \mathscr{Y}^{g,\alpha}_{\tau,T}[\Phi(T,M_T^{\alpha})],
\end{align}
with $\mathscr{Y}^{g,\alpha}$ associated with the obstacle process $\Phi(\cdot,M_\cdot^{\alpha})$.

\begin{lemma}[Admissibility of the family $(\mathcal{Y}^\alpha(\tau))_{\tau \in \mathcal{T}_0}$] The family $(\mathcal{Y}^\alpha(\tau))_{\tau \in \mathcal{T}_0}$ is a \\
square-integrable admissible family.
\end{lemma}

\begin{proof}
For each $S \in \mathcal{T}_0$, $\mathcal{Y}^\alpha(S)$ is an $\mathscr{F}_S$-measurable square-integrable random variable, due to the definitions of the conditional $g$-expectation, essential supremum and infimum operators. Let $S$ and $S'$ be two stopping times in $\mathcal{T}_0$. We set $B:=\{S=S'\}$. We show that $\mathcal{Y}^\alpha(S)=\mathcal{Y}^\alpha(S')$  a.s. on $B$. For all $\theta \in \mathcal{T}_S$, set $\theta_B:=\theta \textbf{1}_{B}+T \textbf{1}_{B^c}$. We clearly have $\theta_B \in \mathcal{T}_{S'}$ and moreover $\theta_B=\theta$ a.s. on $B$.
We also fix $\alpha' \in \mathscr{A}^\alpha_{S'}$ and set $\alpha'_{B}:=\alpha \textbf{1}_{[\![0,S]\!]}+ \alpha' \textbf{1}_{]\!]S,T]\!]}\textbf{1}_{B}$. Clearly $\alpha'_B \in \mathscr{A}_S^\alpha$ and $\alpha'_B=\alpha'$ on $]\!]S',T]\!]\cap B$. By using the fact that $S=S'$ on $B$, as well as several properties of the $g$-expectation,  we obtain
\begin{align}
\textbf{1}_B\mathcal{E}_{S,\theta}^{g}[\Phi(\theta, M_{\theta}^{\alpha'_{B}})]&=\textbf{1}_B\mathcal{E}_{S',\theta}^{g}[\Phi(\theta, M_{\theta}^{\alpha'_{B}})]=\mathcal{E}_{S',\theta}^{g\textbf{1}_B}[\textbf{1}_B\Phi(\theta, M_{\theta}^{\alpha'_{B}})]=\mathcal{E}_{S',\theta}^{g\textbf{1}_B}[\textbf{1}_B\Phi(\theta_B, M_{\theta_B}^{{\alpha}_B^{'}})] \nonumber \\
& =\textbf{1}_B\mathcal{E}_{S',\theta}^{g}[\Phi(\theta_B, M_{\theta_B}^{\alpha'})] \leq \textbf{1}_B \underset{\theta \in \mathcal{T}_{S'}} \esssup \mathcal{E}^g_{S',\theta}[\Phi(\theta,M_\theta^{\alpha'})] {\rm \,\, a.s.}
\end{align}
By taking the essential supremum on $\theta \in \mathcal{T}_{S}$ and then the essential infimum on $\alpha' \in \mathscr{A}_{S'}^\alpha$, we get $\mathcal{Y}^\alpha(S) \leq \mathcal{Y}^{\alpha}(S')$ a.s. on $B$. By interchanging the roles of $S$ and $S'$, the converse inequality follows by the same arguments.
\fproof
\end{proof}

\noindent We now prove the existence of an optimizing sequence in the representation \reff{NewName}.

\begin{lemma}[Optimizing sequence]\label{TH}
Fix $ \tau \in \mathcal{T}_0$, $\theta \in \mathcal{T}_\tau$, $\mu \in \mathbf{L}_0([0,1], \mathscr{F}_\tau)$ and $\alpha \in \mathscr{A}_{\tau, \mu}$. Then there exists a sequence $(\alpha'_n) \subset \mathscr{A}_{\tau, \mu}^{\theta, \alpha}:=\{\alpha' \in \mathscr{A}_{\tau, \mu}, \alpha'\textbf{1}_{[0, \theta)}=\alpha\textbf{1}_{[0,\theta)}\}$ such that $$ \underset{n \rightarrow \infty}\lim \downarrow \mathscr{Y}_{\theta,T}^{g, {\alpha'_n}}[\Phi(T,M_T^{^{\tau,\mu,\alpha'_n}})]=\mathcal{Y}(\theta,M_\theta^{^{\tau,\mu,\alpha}})\ \ \  a.s.$$
\end{lemma}

\begin{proof}
To prove the result, we have to show that the family $\{J(\alpha'):= \mathscr{Y}_{\theta,T}^{g, {\alpha'}}[\Phi(T,M_T^{^{\tau,\mu,\alpha'}})],$ $\alpha' \in \mathscr{A}_{\tau,\mu}^{\theta,\alpha}  \}$ is downward directed {\color{black}; see, for example, \cite{N75}}. Fix $\alpha_1', \alpha_2' \in \mathscr{A}_{\tau,\mu}^{\theta,\alpha}$ and set
$$\Tilde{\alpha}':= \alpha\textbf{1}_{[0,\theta)}+\textbf{1}_{[\theta,T]}(\alpha_1' \textbf{1}_A+\alpha_2' \textbf{1}_{A^c}),$$
where $A:=\{ J(\alpha_1') \leq J(\alpha_2')\} \in \mathscr{F}_\theta$, which implies that $\Tilde{\alpha}' \in \mathscr{A}_{\tau,\mu}^{\theta,\alpha}$ and, since $A \in \mathscr{F}_\theta$,
\begin{align*}
J(\Tilde{\alpha}')&= \mathscr{Y}^{g, {\Tilde{\alpha}'}}_{\theta,T}[\Phi(T,M_T^{^{\tau,\mu,\alpha_{1}'}}) \textbf{1}_A+\Phi(T,M_T^{^{\tau,\mu,\alpha_{2}'}})\textbf{1}_{A^c}]\\
&=\textbf{1}_A\mathscr{Y}^{g, {\alpha_{1}'}}_{\theta,T}[\Phi(T,M_T^{^{\tau,\mu,\alpha_{1}'}})]+\textbf{1}_{A^c}\mathscr{Y}^{g, {\alpha_{2}'}}_{\theta,T}[\Phi(T,M_T^{^{\tau,\mu,\alpha_{2}'}})] \nonumber \\
&= \min(J(\alpha_1'), J(\alpha_2')).
\end{align*}
This gives the desired result. \fproof 
\end{proof}






\noindent We now proceed to show that, for each $\alpha \in \mathscr{A}_0$, the family $(\mathcal{Y}^\alpha(\tau), \,\, \tau \in \mathcal{T}_0)$ is a $\mathscr{Y}^{g,\alpha}$- submartingale family. This is a direct consequence of the following dynamic programming principle.

\begin{theorem}[\textit{Dynamic programming principle}]\label{THH}
The value family   satisfies the following dynamic programming principle: for all $(\tau_1, \tau_2, \alpha) \in \mathcal{T}_0 \times \mathcal{T}_0 \times \mathscr{A}_0$ such that  
$\tau_2\in \mathcal{T}_{\tau_1}$ a.s., we have:
\begin{align}
\mathcal{Y}^{\alpha}({\tau_1})= \essinf_{\alpha' \in \mathscr{A}_{\tau_1}^\alpha} \mathscr{Y}_{\tau_1, \tau_2}^{\,g, {\alpha'}}[\mathcal{Y}^{\alpha'}(\tau_2)] {\rm \,\, a.s.}
\end{align}
\end{theorem}

\begin{proof}\quad\\
Let us first show that
\begin{align*}
\mathcal{Y}^{\alpha}({\tau_1}) \geq  \essinf_{\alpha' \in \mathscr{A}_{\tau_1}^\alpha} \esssup_{\theta \in \mathcal{T}_{\tau_1}} \mathcal{E}^{g}_{\tau_1, \tau_2 \wedge \theta} \left[\mathcal{Y}^{\alpha'}(\tau_2)\textbf{1}_{ \theta\geq\tau_2}+\Phi(\theta, M_{\theta}^{\alpha'})\textbf{1}_{\theta<\tau_2}\right]\,\,  {\rm a.s.}
\end{align*}
Fix $\alpha' \in \mathscr{A}_{\tau_1}^\alpha$. By the flow property for reflected BSDEs we obtain:
\begin{align*}
\mathscr{Y}^{\,g, {\alpha'}}_{\tau_1,T}\left[\Phi(T,M_T^{\alpha'})\right]=\mathscr{Y}^{\,g, {\alpha'}}_{\tau_1,\tau_2}\left[\mathscr{Y}^{g, {\alpha'}}_{\tau_2,T}[\Phi(T, M_T^{\alpha'})]\right]\,\,\, {\rm a.s.}
\end{align*}
Using a comparison argument for reflected BSDEs together with \reff{NewName}, we get:
\begin{align*}
\mathscr{Y}^{\,g, {\alpha'}}_{\tau_1,T}\left[\Phi(T,M_T^{\alpha'})\right] \geq \mathscr{Y}^{\,g, {\alpha'}}_{\tau_1,\tau_2}\left[\mathcal{Y}^{\alpha'}(\tau_2)\right] \,\,\, {\rm a.s.}
\end{align*}
By arbitrariness of $\alpha' \in \mathscr{A}_{\tau_1}^\alpha$, we finally obtain:
\begin{align*}
\mathcal{Y}^\alpha({\tau_1}) \geq \essinf_{\alpha' \in \mathscr{A}_{\tau_1}^\alpha} \mathscr{Y}^{\,g, {\alpha'}}_{\tau_1,\tau_2} \left[\mathcal{Y}^{\alpha'}(\tau_2)\right] {\rm \,\, a.s.}
\end{align*}
Conversely, we prove that
\begin{align}
\mathcal{Y}^{\alpha}({\tau_1}) \leq  \essinf_{\alpha' \in \mathscr{A}_{\tau_1}^\alpha} \esssup_{\theta \in \mathcal{T}_{\tau_1}} \mathcal{E}^{g}_{\tau_1, \tau_2 \wedge \theta} \left[\mathcal{Y}^{\alpha'}(\tau_2)\textbf{1}_{ \theta\geq\tau_2}+\Phi(\theta, M_{\theta}^{\alpha'})\textbf{1}_{\theta<\tau_2}\right].\end{align}
By Lemma $\ref{TH}$, there exists a sequence of controls $\alpha^n \in \mathscr{A}_{\tau_2}^{\alpha'}$ such that:
\begin{align*}
\mathcal{Y}^{\alpha'}({\tau_2})=\underset{n \rightarrow \infty} {\lim}\mathscr{Y}^{g, {\alpha^n}}_{\tau_2,T} \left[ \Phi(T,M_T^{\alpha^n})\right] {\rm \,\, a.s.}
\end{align*}
The continuity of the reflected BSDEs with respect to its terminal condition gives:
\begin{align*}
\mathscr{Y}_{\tau_1, \tau_2}^{g, {\alpha'}}\left[\mathcal{Y}^{\alpha'}({\tau_2})\right]= \underset{n \rightarrow \infty} {\lim}\mathscr{Y}_{\tau_1, \tau_2}^{g, {\alpha'}} \left[\mathscr{Y}^{g, {\alpha^n}}_{\tau_2,T} \left[ \Phi(T,M_T^{\alpha^n})\right]\right] {\rm \,\, a.s.}
\end{align*}
We set:
\begin{align*}
\Tilde{\alpha}_s^{n}:=\alpha'_s \textbf{1}_{s<\tau_2}+\alpha_s^n\textbf{1}_{s\geq \tau_2}.
\end{align*}
The two above relations and the consistency of the operator $\mathscr{Y}^{g, {\alpha'}}$ finally give:
\begin{align*}
\mathscr{Y}_{\tau_1, \tau_2}^{g, {\alpha'}}\left[\mathcal{Y}^{\alpha'}({\tau_2})\right]= \underset{n \rightarrow \infty} {\lim}\mathscr{Y}_{\tau_1, T}^{g, {\Tilde{\alpha}^n}}\left[\Phi(T,M_T^{\Tilde{\alpha}^n})\right] \geq \mathcal{Y}^{\alpha'}({\tau_1}) {\rm \,\, a.s.}
\end{align*}
Now, by arbitrariness of $\alpha' \in \mathscr{A}_\tau^\alpha$, the result follows.
\fproof
\end{proof}

We finally prove the existence of a RCLL process which aggregates the value family $(\mathcal{Y}^\alpha).$
\begin{theorem}[\textit{Existence of a RCLL aggregator process of the value family}]\label{aggregator}
For any $\alpha \in \mathscr{A}_0$, there exists a RCLL process $(\mathcal{Y}^\alpha_t)$ which aggregates the family $(\mathcal{Y}^{\alpha}(\tau), \,\, \tau \in \mathcal{T}_0)$, that is $\mathcal{Y}^\alpha(\tau)=\mathcal{Y}_\tau^\alpha$ a.s., for all $\tau \in \mathcal{T}_0$. 
\end{theorem}

\begin{proof}
 Fix $\alpha \in \mathscr{A}_0$. 
By Theorem \ref{THH}, we get that the family $(\mathcal{Y}^\alpha(\tau))_{\tau \in \mathcal{T}_0}$ is a $\mathscr{Y}^{g,\alpha}$-submartingale family. By Theorem \ref{agreg}, we deduce that there exists an optional process $(\mathcal{Y}^\alpha_t)$ which aggregates the family $(\mathcal{Y}^\alpha(\tau))_{\tau \in \mathcal{T}_0}$ and satisfies $\mathbb{E}[\underset{\tau \in \mathcal{T}_0}{\esssup}(\mathcal{Y}^\alpha_\tau)^2]<\infty$. We can thus use Theorem $\ref{eee1}$, which shows that  $(\mathcal{Y}^\alpha_t)$ admits a $\mathscr{Y}^{g,\alpha}$-Mertens decomposition, giving the existence of its left and right limits. Define the process:
\begin{align*}
\overline{\mathcal{Y}}_t^\alpha:= \underset{s \in (t,T] \downarrow t} {\lim}\mathcal{Y}_s^\alpha, \,\,\, t \in [0,T].
\end{align*}
In order to show that the process $\mathcal{Y}^\alpha$ is indistinguishable of a RCLL process, we need to prove that 
\begin{align*}
\overline{\mathcal{Y}}_\tau^\alpha= \mathcal{Y}_\tau^\alpha {\rm \,\,\,  a.s., \,\, for \,\, all\,\,} \tau \in \mathcal{T}_0.
\end{align*}
Let us introduce $(\tau_n)_{n \in \NB}$, a non-increasing sequence of stopping times with values in $[0,T]$ such that $\tau_n \downarrow \tau$ a.s. as $n \rightarrow +\infty$. By the definition of the process $\overline{\mathcal{Y}}^\alpha$, we have 
\begin{align}\label{rel1}
\overline{\mathcal{Y}}_\tau^\alpha= \underset{n \rightarrow \infty}{\lim} \mathcal{Y}^\alpha_{\tau_n}\  {\rm a.s.}
\end{align}

\noindent \textit{Step 1.} Let us first show the inequality  $\overline{\mathcal{Y}}^\alpha_{\tau}\geq \underset{\alpha' \in \mathscr{A}_\tau^\alpha}{\essinf} \,\mathscr{Y}_{\,\tau, \theta}^{\,g, \alpha'} \left[\Phi(\theta, M_\theta^{\alpha'})\right]=\mathcal{Y}^\alpha_\tau$.\\
By the continuity property of BSDEs with respect to the terminal time and terminal condition, we have  $\underset{n \rightarrow \infty} \lim \mathscr{Y}_{\tau, \tau_n}^{g, \alpha}[\mathcal{Y}_{\tau_n}^\alpha]=\overline{\mathcal{Y}}_\tau^\alpha$ a.s. Then, by Theorem \ref{THH}, we get $\mathscr{Y}_{\tau, \tau_n}^{g, \alpha}[\mathcal{Y}_{\tau_n}^\alpha] \geq \mathcal{Y}_\tau^\alpha$ a.s. The result follows. \\

\noindent \textit{Step 2.} It remains to show that $\overline{\mathcal{Y}}^\alpha_\tau\leq \underset{\alpha' \in \mathscr{A}_\tau^\alpha} {\essinf}\, \mathscr{Y}_{\tau,T}^{g, \alpha'} \left[\Phi(T, M_T^{\alpha'})\right]=\mathcal{Y}^\alpha_\tau {\rm\,\, a.s.}$ \\

\noindent In order to prove it, fix $\alpha' \in \mathscr{A}_\tau^\alpha$ and set
$$\lambda_n:= \left(\frac{M_{\tau_n}^\alpha}{M_{\tau_n}^{\alpha'}} \wedge \frac{1-M_{\tau_n}^\alpha}{1-M_{\tau_n}^{\alpha'}}\right)
{\textbf{1}_{\{M_{\tau_n}^{\alpha'} \notin \{0,1\}\}}} \in [0,1].$$
We set $\alpha'_n:= \alpha \textbf{1}_{[0,\tau_n)}+\lambda_n \alpha' \textbf{1}_{[\tau_n,T]}.$ This implies that $\alpha'_n$ belongs to $\mathscr{A}_{\tau_n}^\alpha$.\\
\noindent Now,  relation \eqref{rel1} together with the $\mathscr{F}_\tau$-measurability of $\underset{n \rightarrow \infty}{\lim} \mathcal{Y}_{\tau_n}^\alpha$ and the continuity of BSDEs with respect to the terminal time and terminal condition give:
\begin{align}\label{e1}
\overline{\mathcal{Y}}_\tau^\alpha {\color{black}=} \,\mathcal{E}^{g}_{\tau,\tau}[\underset{n \rightarrow \infty}{\lim} \mathcal{Y}_{\tau_n}^\alpha]=\underset{n \rightarrow \infty}{\lim} \mathcal{E}^{g}_{\tau,\tau_n}[\mathcal{Y}^\alpha_{\tau_n}] \,\,\, {\rm a.s.}
\end{align}
Using standard results from the optimal stopping theory, there exists an optimal stopping time $\hat{\theta}_n \in \mathcal{T}_{\tau_n}$ for the optimal stopping problem $\underset{\theta \in \mathcal{T}_{\tau_n}}{\esssup}\,\mathcal{E}^g_{\tau_n,\theta}\left[\Phi(\theta, M_\theta^{\alpha_n'})\right].$
We thus derive 
\begin{eqnarray*} 
\mathcal{E}^{g}_{\tau,\tau_n}[\mathcal{Y}^\alpha_{\tau_n}] &\leq& \mathcal{E}^{g}_{\tau, \tau_n}\Big [\esssup_{\theta \in \mathcal{T}_{\tau_n}} \mathcal{E}^g_{\tau_n,\theta}\Big[\Phi\big(\theta, M_\theta^{\alpha_n'})\big]\Big]\\
&=&\mathcal{E}^{g}_{\tau, \tau_n}\left[\mathcal{E}^g_{\tau_n,\hat{\theta}_n}\left[\Phi(\hat{\theta}_n, M_{\hat{\theta}_n}^{\alpha_n'})\right]\right] = \mathcal{E}_{\tau,\hat{\theta}_n}^g\left[\Phi(\hat{\theta}_n, M_{\hat{\theta}_n}^{\alpha_n'})\right] {\rm a.s.},
\end{eqnarray*}
where the first inequality follows by admissibility of the control $\alpha_n'.$
Furthermore, we get 
\begin{align}\label{i1}
\mathcal{E}_{\tau,\hat{\theta}_n}^g\left[\Phi(\hat{\theta}_n, M_{\hat{\theta}_n}^{\alpha_n'})\right]=\mathcal{E}_{\tau,\hat{\theta}_n}^g\left[\Phi(\hat{\theta}_n, M_{\hat{\theta}_n}^{\alpha_n'})\right]-\mathcal{E}_{\tau,\hat{\theta}_n}^g\left[\Phi(\hat{\theta}_n, M_{\hat{\theta}_n}^{\alpha'})\right]+\mathcal{E}_{\tau,\hat{\theta}_n}^g\left[\Phi(\hat{\theta}_n, M_{\hat{\theta}_n}^{\alpha'})\right].
\end{align}
Since $\hat{\theta}_n \in \mathcal{T}_{\tau_n} \subset \mathcal{T}_{\tau}$, we have
\begin{align}\label{i2}
\mathcal{E}_{\tau,\hat{\theta}_n}^g\left[\Phi(\hat{\theta}_n, M_{\hat{\theta}_n}^{\alpha_n'})\right] \leq \mathcal{E}_{\tau,\hat{\theta}_n}^g\left[\Phi(\hat{\theta}_n, M_{\hat{\theta}_n}^{\alpha_n'})\right]-\mathcal{E}_{\tau,\hat{\theta}_n}^g\left[\Phi(\hat{\theta}_n, M_{\hat{\theta}_n}^{\alpha'})\right]+\esssup_{\theta \in \mathcal{T}_\tau} \mathcal{E}_{\tau,\theta}^g\left[\Phi({\theta}, M_{{\theta}}^{\alpha'})\right] {\rm \,\, a.s.}
\end{align}
Now, by using the {\it a priori} estimates with BSDEs we have:
\begin{align}\label{i3}
\mathbb{E}\left[\left| \mathcal{E}_{\tau,\hat{\theta}_n}^g\left[\Phi(\hat{\theta}_n, M_{\hat{\theta}_n}^{\alpha_n'})\right]-\mathcal{E}_{\tau,\hat{\theta}_n}^g\left[\Phi(\hat{\theta}_n, M_{\hat{\theta}_n}^{\alpha'})\right] \right|^2 \right] \leq C \mathbb{E} \left[ \left(\Phi(\hat{\theta}_n, M_{\hat{\theta}_n}^{\alpha_n'})-\Phi(\hat{\theta}_n, M_{\hat{\theta}_n}^{\alpha'})\right)^2\right]
\end{align}
$$ \leq C \mathbb{E}\left[\sup_{0 \leq t \leq T}\left(\Phi(t, M_{t}^{\alpha_n'})-\Phi(t, M_{t}^{\alpha'})\right)^2\right].$$
One can easily show that $M_T^{\alpha_n'} \rightarrow M_T^{\alpha'}$ a.s. when $n \rightarrow \infty$ {\color{black}(see the proof of Step 1 b) of Proposition 4.2 in \cite{BER15})}. By applying  Doob's inequality and Lebesgue's Theorem, and using the uniform continuity of $\Phi$, we derive that  
\begin{align}\label{i4}
\mathbb{E}\left[\underset{0 \leq t \leq T}{\sup}\left(\Phi(t, M_{t}^{\alpha_n'})-\Phi(t, M_{t}^{\alpha'})\right)^2\right] \rightarrow 0 {\rm \,\,when\,\,} n \rightarrow \infty.
\end{align}
\noindent Finally, by combining $\eqref{e1}, \eqref{i1}, \eqref{i2}, \eqref{i3},\eqref{i4}$ and taking the limit in $n$, the result follows.
\fproof
\end{proof}

We now prove the representation of the minimal $t$-values process $\mathcal{Y}^\alpha_\cdot$ in terms of a reflected  Backward SDE.

\begin{Theorem}[Reflected BSDE representation of the minimal $t$-values process]\label{BSDE representation}

There exists a family $(\mathcal{Z}^\alpha,\mathcal{A}^\alpha, \mathcal{K}^\alpha)_{\alpha \in \mathscr{A}_0}$ $\subset \mathbf{H}_2 \times \mathbf{K}_2 \times \mathbf{K}_2$
such that, for all $\alpha \in \mathscr{A}_0$, we have, for all $0 \leq t \leq T$,

\begin{align}
 &\mathcal{Y}_t^\alpha=\Phi(T,M_T^\alpha)+\int_t^Tg(s,\mathcal{Y}_s^\alpha,\mathcal{Z}_s^\alpha)ds-\int_t^T\mathcal{Z}_s^\alpha dW_s+\mathcal{K}^\alpha_t-\mathcal{K}_T^\alpha-\mathcal{A}^\alpha_t+\mathcal{A}_T^\alpha; \label{eqqq2} \\
& \mathcal{Y}_t^{\alpha} \geq \Phi(t, M_t^{\alpha}) {\rm \,\, a.s.};  \label{eqq3}\\
&\int_0^{T}\left(\mathcal{Y}_{s^-}^\alpha-\Phi(s,M_{s^-}^\alpha)\right)d\mathcal{A}_s^\alpha=0 \text{\,\, a.s.};\,\,\, d\mathcal{A}^\alpha \perp d \mathcal{K}^\alpha; \label{eqqq4} \\
& \underset{\alpha' \in \mathscr{A}_\tau^\alpha}{\essinf}\,\mathbb{E}[\int_\tau^T \exp (\delta_s^{\tau, \alpha'})d(A_s^{\alpha'}-\mathcal{A}_s^{\alpha'}+\mathcal{K}_s^{\alpha'})]=0 {\rm \,\, a.s., \,\, for\,\, all \,\,} \tau \in \mathcal{T}_0;   \label{eqqq5} \\
&
(\mathcal{Y}^\alpha,\mathcal{Z}^\alpha,\mathcal{K}^\alpha,\mathcal{A}^\alpha)\textbf{1}_{[\![0,\tau]\!]}=(\mathcal{Y}^{\bar{\alpha}},\mathcal{Z}^{\bar{\alpha}},\mathcal{K}^{\bar{\alpha}},\mathcal{A}^{\bar{\alpha}})\textbf{1}_{[\![0,\tau]\!]}, \,\, \forall \tau \in \mathcal{T}_0, \,\bar\alpha \in \mathscr{A}_\tau^{{\alpha}} \label{eqqq6}.\\
\nonumber
\end{align}
where 
$$\delta_s^{t, \alpha}:=\int_t^s \{\beta^{\alpha}_u dW_u+(\lambda^{\alpha}_u-\dfrac{(\beta^{\alpha}_u)^2}{2})du\},$$
with
$$\lambda^\alpha_s:= \dfrac{g(s,\mathcal{Y}_s^{\alpha}, \mathcal{Z}_s^{\alpha})-g(s,Y_s^{\alpha}, \mathcal{Z}_s^{\alpha})}{\mathcal{Y}_s^{\alpha}-Y_s^{\alpha}} \textbf{1}_{\{\mathcal{Y}_s^{\alpha}-Y_s^{\alpha} \neq 0\}};$$
$$ \beta^\alpha_s:= \dfrac{g(s,Y_s^{\alpha}, \mathcal{Z}_s^{\alpha})-g(s,Y_s^{\alpha}, Z_s^{{\alpha}})}{|\mathcal{Z}_s^{\alpha}-Z_s^{\alpha}|^2}(\mathcal{Z}_s^{\alpha}-Z_s^{\alpha}) \textbf{1}_{\{\mathcal{Z}_s^{\alpha}-Z_s^{\alpha} \neq 0\}},$$
and $(Y^{\alpha}, Z^{\alpha},A^{\alpha})$ the solution of the reflected BSDE with driver $g$ and obstacle $\Phi(\cdot,M_\cdot^{\alpha})$.
Moreover, $(\mathcal{Y}^\alpha, \mathcal{Z}^\alpha, \mathcal{A}^\alpha, \mathcal{K}^\alpha)_{\alpha \in \mathscr{A}_0} \in \mathbf{S}_2 \times \mathbf{H}_2 \times (\mathbf{K}_2)^2$ is the unique family satisfying $\eqref{eqqq2}-\eqref{eqqq6}$.

\end{Theorem}

\begin{proof}
First note that for $(\alpha, \tau) \in \mathscr{A}_0 \times \mathcal{T}_0$, we have $\mathscr{A}_\cdot^{\alpha'}=\mathscr{A}_\cdot^\alpha$ on $[\![0,\tau]\!]$ for $\alpha' \in \mathscr{A}_\tau^\alpha$. The definition of $\mathcal{Y}^\alpha$ implies that $\mathcal{Y}^\alpha \textbf{1}_{[0,\tau]}=\mathcal{Y}^{\alpha'} \textbf{1}_{[0,\tau]}$ for $\alpha' \in \mathscr{A}_\tau^\alpha$. Fix $\tau \in \mathcal{T}_0$ and $\alpha \in \mathscr{A}_0$. By Theorem \ref{aggregator},  we get that the process $\mathcal{Y}^\alpha_\cdot$ is a RCLL $\mathscr{Y}^{g,\alpha}$-submartingale, and therefore we can apply the $\mathscr{Y}^{g,\alpha}$-Doob-Meyer decomposition provided in Theorem \ref{DMRCLL} and obtain the existence of $(\mathcal{Z}^\alpha,\mathcal{A}^\alpha,\mathcal{K}^\alpha) \in \mathbf{H}_2 \times (\mathbf{K}_2)^2$ such that, for $t \in [0,T]$,
\begin{align}
\begin{cases}
\mathcal{Y}^\alpha_t= \Phi(T,M_T^\alpha)+\displaystyle\int_t^Tg(s,\mathcal{Y}_s^\alpha,\mathcal{Z}^\alpha_s)ds+\mathcal{A}_T^\alpha-\mathcal{A}_t^\alpha-\mathcal{K}_T^\alpha+\mathcal{K}_t^\alpha-\displaystyle\int_t^T\mathcal{Z}_s^\alpha dW_s;\\
\mathcal{Y}_t^\alpha \geq \Phi(t,M_t^\alpha)  {\rm \,\, a.s.\,\,}; \nonumber \\
\displaystyle\int_0^T (\mathcal{Y}^\alpha_{s^-}-\Phi(s^-, M_{s^-}^\alpha))d\mathcal{A}_s^\alpha=0;\,\,
d\mathcal{A}_s^\alpha \perp d\mathcal{K}_s^\alpha. \nonumber
\end{cases}
\end{align}
By the uniqueness of the representation of a semimartingale and since the measures $d\mathcal{A}^\alpha$ and $d\mathcal{K}^\alpha$ (resp. $d\mathcal{A}^{\bar{\alpha}}$ and $d\mathcal{K}^{\bar{\alpha}}$, for all $\bar{\alpha} \in \mathscr{A}_\tau^{\alpha}$) are mutually singular, we derive that $(\mathcal{Y}^\alpha,\mathcal{Z}^\alpha,\mathcal{K}^\alpha,\mathcal{A}^\alpha)\textbf{1}_{[\![0,\tau]\!]}=(\mathcal{Y}^{\bar{\alpha}},\mathcal{Z}^{\bar{\alpha}},\mathcal{K}^{\bar{\alpha}},\mathcal{A}^{\bar{\alpha}})\textbf{1}_{[\![0,\tau]\!]}, \,\,  \forall \bar{\alpha} \in \mathscr{A}_\tau^{\alpha}$.
 It remains to show the minimality condition \eqref{eqqq5}.\\
To do so, let us first consider an arbitrary control $\bar{\alpha} \in \mathscr{A}_{\tau}^\alpha$ and $(Y^{\bar{\alpha}}, Z^{\bar{\alpha}}, A^{\bar{\alpha}})$ the solution of the following  reflected BSDE:
\begin{align*}
\begin{cases}
Y_t^{\bar{\alpha}}= \Phi(T,M_T^{\bar{\alpha}})+\displaystyle\int_t^T g(s, Y_s^{\bar{\alpha}}, Z_s^{\bar{\alpha}})ds-\displaystyle\int_t^T Z_s^{\bar{\alpha}}dW_s+A_T^{\bar{\alpha}}-A_t^{\bar{\alpha}};\\
Y_t^{\bar{\alpha}} \geq \Phi(t, M_t^{\bar{\alpha}})   {\rm \,\, a.s.\,\,} \,\, 0 \leq t \leq T; \\
\displaystyle\int_0^T (Y_{s_-}^{\bar{\alpha}}-\Phi(s^{-}, M_{s_-}^{\bar{\alpha}}))dA_s^{\bar{\alpha}}=0.
\end{cases}
\end{align*}
Using a classical linearization procedure, we obtain:
\begin{align}
Y_\tau^{\bar{\alpha}}-\mathcal{Y}^{\bar{\alpha}}_\tau = \mathbb{E}_\tau[\int_\tau^{T} \exp (\delta_s^{\tau, \bar{\alpha}})(dA_s^{\bar{\alpha}}-d\mathcal{A}_s^{\bar{\alpha}}+d\mathcal{K}_s^{\bar{\alpha}})] {\rm \,\, a.s.}
\end{align}
We take now the $\essinf$ on $\bar{\alpha} \in \mathscr{A}_\tau^\alpha$ and using the definition of the value function $\mathcal{Y}^{\alpha}$ and the fact that $\mathcal{Y}^\alpha \textbf{1}_{[0,\tau]}=\mathcal{Y}^{\bar{\alpha}} \textbf{1}_{[0,\tau]}$ for $\bar{\alpha} \in \mathscr{A}_\tau^\alpha$, the minimality condition follows.

We now show the uniqueness of the family.
Let $(\Tilde{Y}^\alpha,\Tilde{Z}^\alpha,\Tilde{K}^\alpha,\Tilde{A}^\alpha)$ be a solution of \eqref{eqqq2}-\eqref{eqqq6}. Fix $\tau \in \mathcal{T}_0$ and $\bar{\alpha} \in \mathscr{A}_\tau^{\alpha}$ and denote by $(Y^{\bar{\alpha}}, Z^{\bar{\alpha}},A^{\bar{\alpha}})$ the solution of the reflected BSDE with driver $g$ and obstacle $\Phi(\cdot,M_\cdot^{\bar{\alpha}})$.
By using the same linearization procedure, we obtain
\begin{align}
Y_\tau^{\bar{\alpha}}-\Tilde{Y}_\tau^{\bar{\alpha}}=\mathbb{E}[\int_{\tau}^T \exp (\delta_s^{\tau, \bar{\alpha}})d(A^{\bar{\alpha}}_s-\Tilde{A}_s^{\bar{\alpha}}+\Tilde{K}_s^{\bar{\alpha}})] {\rm \,\, a.s.}
\end{align}
The minimality condition $\eqref{eqqq5}$, together with $\eqref{eqqq6}$ and the definition of $\mathcal{Y}^\alpha$ imply that $\tilde{Y}_\tau^{\alpha}=\mathcal{Y}_\tau^\alpha$ a.s. By the uniqueness of the representation of a semimartingale and since the measures $d\mathcal{A}^\alpha$ and $d\mathcal{K}^\alpha$ (resp. $d\Tilde{A}^{\bar{\alpha}}$ and $d\Tilde{K}^{\bar{\alpha}}$) are mutually singular, we get the {\color{black}uniqueness }result.
\fproof
\end{proof}

\begin{Remark}
 Since  the process $A^\alpha-\mathcal{A}^\alpha+\mathcal{K}^\alpha$ is in general not non-decreasing, notice that we cannot derive  a formulation only involving $\mathcal{A}^\alpha$, $A^\alpha$ and $\mathcal{K}^\alpha$, as for non-reflected BSDEs with weak terminal condition. We point out that in the case whenever $\Phi=-\infty$ implying no reflection, the processes $\mathcal{A}^\alpha$ and $A^\alpha$ become 0 for all $\alpha \in \mathscr{A}_0$. Hence, the minimality condition is indeed equivalent to 
\begin{align}
\essinf_{\alpha' \in \mathscr{A}_\tau^\alpha} \mathbb{E}_\tau \left[\mathcal{K}_T^{\alpha'}-\mathcal{K}_\tau^{\alpha'}\right] =0 {\rm\,\, a.s.}
\end{align}
corresponding to the minimality condition presented in Bouchard et al. \cite{BER15} for BSDEs with weak terminal condition and in Soner et al. \cite{STZ12} for second order BSDEs.

The need to depart from the "standard" minimality condition has also been pointed out in Matoussi et al.  \cite{MPZ17}, when dealing with second order reflected BSDEs, as well as  in Popier et al. \cite{PZ18}, when dealing with 2BSDEs under a monotonicity condition.
\end{Remark}



\section{Appendix}\label{Sec4}
\paragraph{\textit{BSDEs with weak constraints} and a related game problem}

In this section, we study a related game problem. We show that, given a threshold process $(M_t^{\alpha})$, the minimal initial process $\mathcal{Y}^\alpha$ corresponds to the value of an optimal stopping problem. More precisely, we provide some conditions under which one can interchange the inf and the sup and deduce the existence of a saddle point. This problem is in general non-trivial, and the additional complexity in our case is due to the presence of the control $\alpha$ in the obstacle $\Phi(t, M_t^\alpha)$.\\

Let $S \in \mathcal{T}_0$ and $\alpha \in \mathscr{A}_0$. Define the \textit{first value function} at time $S$ by
\begin{align}\label{valfunct1}
\overline{\mathcal{Y}}^{\alpha}(S):=\essinf_{{\alpha}^{'} \in \mathscr{A}_{S}^{\alpha}} \esssup_{\tau \in \mathcal{T}_S} \mathcal{E}^{g}_{S,\tau}[\Phi(\tau,M_{\tau}^{\alpha^{'}})].
\end{align}
and \textit{the second value function} at time $S$ by
\begin{align}\label{valfunct2}
\underline{\mathcal{Y}}^{\alpha}(S):= \esssup_{\tau \in \mathcal{T}_S} \essinf_{{\alpha}^{'} \in \mathscr{A}_{S}^{\alpha}} \mathcal{E}^{g}_{S,\tau}[\Phi(\tau,M_{\tau}^{\alpha^{'}})].
\end{align}
By definition, we say that there exists a \textit{value function} at time $S$ for the game problem if  $\overline{\mathcal{Y}}^{\alpha}(S)=\underline{\mathcal{Y}}^{\alpha}(S)$ a.s.

We recall the definition of a $S-$\textit{saddle point}.

\begin{definition}[S-saddle point]\label{def} Let $S \in \mathcal{T}_0$. A pair $(\tau_S^*, \alpha_S^*) \in \mathcal{T}_S \times \mathscr{A}_{S}^{\alpha}$ is called a $S$-\textit{saddle point} if

\begin{itemize}
\item[(i)] $\overline{\mathcal{Y}}^{\alpha}(S)=\underline{\mathcal{Y}}^{\alpha}(S)$\,\, a.s.
\item[(ii)] The essential infimum in \eqref{valfunct1} is attained at $\alpha_S^*$
\item[(iii)] The essential supremum in \eqref{valfunct2} is attained at $\tau_S^*$.
\end{itemize}
\end{definition}

In order to prove the existence of a $S$-saddle point, we need to make the following convexity assumption on the driver.

\begin{Assumption}\label{con} For all $(\lambda, m_1,m_2,t,y_1,y_2,z_1,z_2)\in [0,1]\times [0,1]^2 \times [0,T]\times \mathbf{R}^2 \times (\mathbf{R}^d)^2$,
\begin{align*}
g(t,\lambda y_1+ (1-\lambda)y_2, \lambda z_1+ (1-\lambda)z_2) &\leq \lambda g(t,y_1,z_1)+ (1-\lambda)g(t,y_2,z_2)\, {\rm \,\,a.s.}
\end{align*}
\end{Assumption}

Let us now give the main result of this section. 
\begin{theorem} \label{SSS}
1. Assume that  $g(t,\omega,y,z) \geq 0$, for all $(t,\omega,y, z) \in [0,T] \times \Omega \times \textbf{R} \times \textbf{R}^d$ and suppose that $\Phi$  is non-decreasing with respect to $t$ and convex with respect to $m$. Then the game problem admits a value function, that is
\begin{align}\label{game}
\overline{\mathcal{Y}}^\alpha(S)=\underline{\mathcal{Y}}^\alpha(S) \,\,\, {\rm \,\,a.s.,\,\, for\,\, all\,\,} S \in \mathcal{T}_0.
\end{align}

2. Assume that  $g(t,\omega,y,z) \leq 0$, for all $(t,\omega,y, z) \in [0,T] \times \Omega \times \textbf{R} \times \textbf{R}^d$ and suppose that $\Phi$  is non-increasing with respect to $t$ and concave with respect to $m$. Then the game problem admits a value function, that is
\begin{align}\label{game2}
\overline{\mathcal{Y}}^\alpha(S)=\underline{\mathcal{Y}}^\alpha(S) \,\,\, {\rm \,\, a.s.,\,\, for\,\, all\,\,} S \in \mathcal{T}_0.
\end{align}

3. {\color{black}Under the Assumptions of point 1 and Assumption \ref{con}}, there exists a $S$-saddle point for the game problem $\eqref{valfunct1}-\eqref{valfunct2}$ in the sense of Definition \ref{def}.
\end{theorem}

\begin{proof}
1.  Fix $S \in \mathcal{T}_0$.
First note that 
$$\esssup_{\theta \in \mathcal{T}_S}  \essinf_{\alpha' \in \mathscr{A}_S^\alpha} \mathcal{E}_{S,\theta}^{g}\left[\Phi(\theta,M_\theta^{\alpha'})]\right] \leq \essinf_{\alpha' \in \mathscr{A}_S^\alpha} \esssup_{\theta \in \mathcal{T}_S} \mathcal{E}^{g}_{S,\theta}\left[\Phi(\theta,M_\theta^{\alpha'})\right] {\rm \,\, a.s.} $$

It remains to show the converse inequality.

\noindent Fix $\theta \in \mathcal{T}_S$ and $\alpha' \in \mathscr{A}_S^\alpha$. By the flow property for nonlinear BSDEs, we get
\begin{align*}
\mathcal{E}_{S,T}^{g}[\Phi(T,M_T^{\alpha'})]= \mathcal{E}_{S,\theta}^{g}[\mathcal{E}_{\theta,T}^{g}[\Phi(T,M_T^{\alpha'})]] {\rm \,\, a.s.}
\end{align*}
Applying  the comparison theorem for BSDEs and using the assumption on the driver $g$, we derive
\begin{align}\label{eq1}
\mathcal{E}_{S,\theta}^{g}\left[\mathcal{E}_{\theta,T}^{g}[\Phi(T,M_T^{\alpha'})]\right] \geq \mathcal{E}^{g}_{S,\theta}\left[\mathbb{E}[\Phi(T,M_T^{\alpha'})|\mathscr{F}_\theta]\right] {\rm \,\, a.s.}
\end{align}
The above relation, together with  the properties of the map $\Phi$ and the conditional Jensen inequality implies that 
\begin{align}\label{eq2}
\mathcal{E}^{g}_{S,\theta}\left[\mathbb{E}[\Phi(T,M_T^{\alpha'})|\mathscr{F}_\theta]\right] \geq \mathcal{E}^{g}_{S,\theta}\left[\mathbb{E}[\Phi(\theta,M_T^{\alpha'})|\mathscr{F}_{\theta}]\right] \geq 
\mathcal{E}^{g}_{S,\theta}\left[\Phi(\theta,\mathbb{E}[M_T^{\alpha'}|\mathscr{F}_{\theta}])\right] {\rm \,\, a.s.}
\end{align}
The martingale property of $M^{\alpha'}$ implies that
\begin{align}\label{eq3}
\mathcal{E}^{g}_{S,\theta}\left[\Phi(\theta,\mathbb{E}[M_T^{\alpha'}|\mathscr{F}_{\theta}])\right]=\mathcal{E}^{g}_{S,\theta}\left[\Phi(\theta,M_\theta^{\alpha'})\right] {\rm \,\, a.s.}
\end{align}
Combining \eqref{eq2} and \eqref{eq3}, we get
\begin{align*}
\mathcal{E}_{S,T}^{g}\left[\Phi(T,M_T^{\alpha'})]\right] \geq \mathcal{E}^{g}_{S,\theta}\left[\Phi(\theta,M_\theta^{\alpha'})\right]  {\rm \,\, a.s.}
\end{align*}
By taking first the essential supremum on $\theta \in \mathcal{T}_S$ and then the essential infimum on $\alpha' \in \mathscr{A}_S^\alpha$, it follows that 
\begin{equation}\label{inequa}
\essinf_{\alpha' \in \mathscr{A}_S^\alpha} \mathcal{E}_{S,T}^{g}[\Phi(T,M_T^{\alpha'})] \geq \essinf_{\alpha' \in \mathscr{A}_S^\alpha} \esssup_{\theta \in \mathcal{T}_S} \mathcal{E}^{g}_{S,\theta}\left[\Phi(\theta,M_\theta^{\alpha'})\right] {\rm \,\, a.s.}
\end{equation}
Besides, we clearly have
\begin{align*}
\essinf_{\alpha' \in \mathscr{A}_S^\alpha} \mathcal{E}_{S,T}^{g}[\Phi(T,M_T^{\alpha'})] \leq \esssup_{\theta \in \mathcal{T}_S}  \essinf_{\alpha' \in \mathscr{A}_S^\alpha} \mathcal{E}_{S,\theta}^{g}[\Phi(\theta,M_\theta^{\alpha'})] {\rm \,\, a.s.}
\end{align*}
The last two inequalities allow to derive
$$\esssup_{\theta \in \mathcal{T}_S}  \essinf_{\alpha' \in \mathscr{A}_S^\alpha} \mathcal{E}_{S,\theta}^{g}\left[\Phi(\theta,M_\theta^{\alpha'})]\right] \geq \essinf_{\alpha \in \mathscr{A}_S^\alpha} \esssup_{\theta \in \mathcal{T}_S} \mathcal{E}^{g}_{S,\theta}\left[\Phi(\theta,M_\theta^{\alpha'})\right] {\rm \,\, a.s.} $$
and \reff{game} follows.
\noindent

\vspace{5mm}

3. We now show the existence of a $S$-saddle point, under the additional assumption that $g$ is convex with respect to $(y,z)$, that is, Assumption \ref{con} holds.

We start by proving the existence of an optimal control  for the problem \eqref{valfunct1}. 
By Lemma \ref{TH} , there exists a sequence  of controls $ (\alpha^n_S)_n$ belonging to  $\mathscr{A}_S^\alpha$ such that 
\begin{equation}\label{downusa}
\overline{\mathcal Y}^{\alpha}(S)=\lim_{n\to\infty} \downarrow \esssup_{\theta \in \mathcal{T}_S}{\mathcal E}^g_{S,\theta}[\Phi(\theta,M^{\alpha^n_S}_\theta)]\,\, {\rm \,a.s. }
\end{equation}
As the sequence $(M^{\alpha^n_S}_T)_n$ is bounded in $[0,1]$, one can find sequences of nonnegative real numbers $(\lambda^n_i)_{i\geq n}$ with $\sum_{i\geq n} \lambda^n_i=1$, such that only a finite number of $\lambda^n_i$'s do not vanish, for each $n$, and such that the sequence of convex combinations $(\tilde{M}^n_T)_n$ given by 

\begin{equation*}
\tilde{M}^n_T:=\sum_{i\geq n} \lambda^n_i M^{\alpha^i_S}_T
\end{equation*}
converges a.s. to some $\bar M_T$. By dominated convergence, the convergence holds in $\textbf{L}_2$, in particular $\mathbb{E}[\bar M_T]=m_0$ and the martingale representation theorem gives the existence of a control ${\color{black}{\alpha_S^*}}$ such that $\bar M_T=M^{m_0, {\color{black}{\alpha_S^*}}}_T$. Let us define $\bar M_t:=\mathbf E_t[\bar M_T]=M^{m_0, {\color{black}{\alpha_S^*}}}_t$. 
By the definition of $(\tilde{M}_t^n)$ and the fact that $(M^{\alpha^n_S}_t)_n$ are martingales, we obtain that, for all $\theta \in \mathcal{T}_S$, $\tilde{M}_\theta^n= \sum_{i\geq n} \lambda^n_i M^{\alpha^i_S}_\theta$ a.s., in particular, $\tilde{M}_S^n= \sum_{i\geq n} \lambda^n_i M^{\alpha^i_S}_S=M^{\alpha}_S$, because $ (\alpha^n_S)_n$ belongs to  $\mathscr{A}_S^\alpha$. Thus, by the $\textbf{L}_2$ convergence, we have that ${\color{black}{\alpha_S^*}}\in\mathscr{A}_S^\alpha$. \\\\
Moreover, since $\Phi$ and $g$ are convex, we have 
\begin{equation*}
\sum_{i\geq n} \lambda^n_i {\mathcal E}^g_{\tau,\theta}[\Phi(\theta,M^{\alpha^i_S}_\theta)]\geq {\mathcal E}^g_{\tau,\theta}[\Phi(\theta,\tilde{M}^n_\theta)] {\rm \,\, a.s.}
\end{equation*}
We thus obtain 
\begin{eqnarray}
{\mathcal Y}^{n}(S)&:=&\underset{i \geq n} \sum \lambda_i^n \esssup_{\theta \in \mathcal{T}_S} {\mathcal E}^g_{S,\theta}[\Phi(\theta,M^{\alpha^i_S}_\theta)] \nonumber\\
&\geq& \esssup_{\theta \in \mathcal{T}_S} \left(\sum_{i\geq n} \lambda^n_i {\mathcal E}^g_{S,\theta}[\Phi(\theta,M^{\alpha^i_S}_\theta)]\right)\geq \esssup_{\theta \in \mathcal{T}_S}{\mathcal E}^g_{S,\theta}[\Phi(\theta,\tilde{M}^n_\theta)] {\rm \,\, a.s.}\label{convexity1}
\end{eqnarray}
Then (\ref{downusa}) implies that $\mathcal {Y}^n(S) \to \overline{\mathcal Y}^{\alpha}(S) $ ${\rm } {\rm \,\, a.s.}$ \\
\noindent  Let us now show that 
\begin{equation}\label{eqqs}
\underset{\theta \in \mathcal{T}_S} \esssup\, {\mathcal E}^g_{S,\theta}[\Phi(\theta,\tilde{M}^n_\theta)] \rightarrow \underset{\theta \in \mathcal{T}_S} \esssup\,{\mathcal E}^g_{S,\theta}[\Phi(\theta,\bar{M}_\theta)] {\rm \,\, a.s.}
\end{equation}


The {\it a priori} estimates on BSDEs give:
\begin{align*}
&\left|\underset{\theta \in \mathcal{T}_S}{ \esssup}{\mathcal E}^g_{S,\theta}[\Phi(\theta,\tilde{M}^n_\theta)]-\underset{\theta\in \mathcal{T}_S}{ \esssup}{\mathcal E}^g_{S,\theta}[\Phi(\theta,\bar{M}_\theta)]\right| \leq \underset{\theta \in \mathcal{T}_S}{ \esssup} \left|{\mathcal E}^g_{S,\theta}[\Phi(\theta,\tilde{M}^n_\theta)]-{\mathcal E}^g_{S,\theta}[\Phi(\theta,\bar{M}_\theta)]  \right| \nonumber \\
& \leq C \underset{\theta \in \mathcal{T}_S}{ \esssup}\,\, \mathbb{E}_S\left[\left(\Phi(\theta, \Tilde{M}_\theta^n)-\Phi(\theta, \bar{M}_\theta)\right)^2\right]^{\frac{1}{2}} \leq C \mathbb{E}_S\left[\sup_{0 \leq t \leq T} \left(\Phi(t, \Tilde{M}_t^n)-\Phi(t, \bar{M}_t)\right)^2\right]^{\frac{1}{2}}{\rm \,\, a.s.,}
\end{align*}
with $C$ a constant depending on $T$ and the Lipschitz constant of the driver $g$.

The Doob maximal inequality together with the uniform continuity of $\Phi$ with respect to $t$ and $m$ imply the convergence to $0$, up to a subsequence,  of the RHS term of the above inequality. Hence, we obtain $\eqref{eqqs}$. From \eqref{convexity1} and $\eqref{eqqs}$, we derive that $$\overline{\mathcal Y}^{\alpha}(S)\geq \underset{\theta \in \mathcal{T}_S} \esssup\,{\mathcal E}^g_{S,\theta}[\Phi(\theta,\bar{M}_\theta)] {\rm \,\, a.s.},$$
thus ${\color{black}{\alpha_S^*}}$ is an optimal control by the definition of $\overline{\mathcal Y}^{\alpha}(S)$.
{\color{black}  Furthermore, by \eqref{inequa} and the fact that $\overline{\mathcal Y}^{\alpha}(S)= \underline{\mathcal Y}^{\alpha}(S)$ we deduce that $\tau_S^*=T$ is an optimal stopping time for $\underline{\mathcal Y}^{\alpha}(S)$. Therefore, we conclude that the pair $(\tau_S^*, \alpha_S^*)$ is a $S$-saddle point.}
\fproof
\end{proof}

\begin{Remark}
\begin{itemize}
\item {\color{black}The proof of the point 2 of Theorem \ref{SSS} is omitted because it follows similar ideas as in the proof of point 1 of the same theorem.}
\item
We emphasize that the results {\color{black}of Theorem \ref{SSS} (1 and 2) still hold} under different assumptions on the map $\Phi$. Indeed, in the case of a positive driver $g$, one could consider the function $\Phi$ of the form $\Phi(t,\omega,m)=m+h(X_t)$, with $X$ a submartingale process and $h$ a convex function. In the case of a negative driver $g$, the proof still works for a function $\Phi$ of the form $\Phi(t,\omega,m)=m+h(X_t)$, with $X$ a supermartingale process and $h$ a concave function. 
\end{itemize}
\end{Remark}

We also easily observe that the existence of the value function of the game implies that we have the following representation of the minimal process $\mathcal{Y}^\alpha$.

\begin{Corollary} Fix $\theta \in \mathcal{T}_0$ and $\alpha \in \mathscr{A}_0$. Then,  {\color{black}under the Assumptions of Theorem \ref{SSS} point 1 and Assumption \ref{con}}, $\mathcal{Y}^\alpha_\theta$ corresponds to the value of the following optimal stopping problem
\begin{align}
\mathcal{Y}^\alpha_\theta=\underset{\tau \in \mathcal{T}_\theta} \esssup\, \mathcal{X}_\theta^{\alpha,\tau} {\rm a.s.},
\end{align}
where $\mathcal{X}^{\alpha,\tau}_\theta$ corresponds to the \textit{minimal} $\theta$-initial supersolution of the BSDE with weak terminal condition at time $\tau$.
\end{Corollary}

@article{peng:99,
	Author = {Peng, S.},
	Journal = {Probab. Theory Relat. Fields},
	Language = {English},
	Number = 4,
	Pages = {473-499},
	Title = {{Monotonic limit theorem of BSDE and nonlinear decomposition theorem of Doob-MeyerF's type.}},
	Volume = 113,
	Year = 1999}
	
	@book{N75,
  title={Discrete-parameter Martingales},
  author={Neveu, J.},
  isbn={9780444107084},
  lccn={74079241},
  series={Mathematical Studies},
  url={https://books.google.fr/books?id=t8cUvgAACAAJ},
  year={1975},
  publisher={North-Holland}
}
	@article {CP2000,
    AUTHOR = {Chen, Zengjing and Peng, Shige},
     TITLE = {A general downcrossing inequality for {$g$}-martingales},
   JOURNAL = {Statist. Probab. Lett.},
  FJOURNAL = {Statistics \& Probability Letters},
    VOLUME = {46},
      YEAR = {2000},
    NUMBER = {2},
     PAGES = {169--175},
      ISSN = {0167-7152},
     CODEN = {SPLTDC},
   MRCLASS = {60H10 (60G44)},
  MRNUMBER = {1748870 (2000m:60068)},
       DOI = {10.1016/S0167-7152(99)00102-9},
       URL = {http://dx.doi.org/10.1016/S0167-7152(99)00102-9},
}
@incollection {DL82,
    AUTHOR = {Dellacherie, C. and Lenglart, E.},
     TITLE = {Sur des probl\`emes de r\'egularisation, de recollement et
              d'interpolation en th\'eorie des processus},
 BOOKTITLE = {Seminar on {P}robability, {XVI}},
    SERIES = {Lecture Notes in Math.},
    VOLUME = {920},
     PAGES = {298--313},
 PUBLISHER = {Springer, Berlin},
      YEAR = {1982},
   MRCLASS = {60G40 (60G07)},
  MRNUMBER = {658692},
MRREVIEWER = {Mich\~A\"{\ }le Mastrangelo-Dehen},
       DOI = {10.1007/BFb0092793},
       URL = {http://dx.doi.org/10.1007/BFb0092793},
}	
@article {LX05,
    AUTHOR = {Lepeltier, J.-P. and Xu, M.},
     TITLE = {Penalization method for reflected backward stochastic
              differential equations with one r.c.l.l. barrier},
   JOURNAL = {Statist. Probab. Lett.},
  FJOURNAL = {Statistics \& Probability Letters},
    VOLUME = {75},
      YEAR = {2005},
    NUMBER = {1},
     PAGES = {58--66},
      ISSN = {0167-7152},
   MRCLASS = {60H10 (60H20)},
  MRNUMBER = {2185610},
MRREVIEWER = {A. I. Dale},
       DOI = {10.1016/j.spl.2005.05.016},
       URL = {http://dx.doi.org/10.1016/j.spl.2005.05.016},
}

 @article{MPZ17,
author = {Anis Matoussi and Dylan Possama{\"\i} and Chao Zhou},
title = {{Corrigendum for  Second-order reflected backward stochastic differential equations  and Second-order BSDEs with general reflection and game options under uncertainty}},
volume = {31},
journal = {The Annals of Applied Probability},
number = {3},
publisher = {Institute of Mathematical Statistics},
pages = {1505 -- 1522},
keywords = {2BSDEs, reflections, Skorokhod condition},
year = {2021},
doi = {10.1214/20-AAP1622},
URL = {https://doi.org/10.1214/20-AAP1622}
}

@article{PZ18,
  title={Second order BSDE under monotonicity condition and liquidation problem under uncertainty},
  author={Popier, A. and Zhou, C.},
  journal = {Ann. Appl. Probab.},
  FJOURNAL = {The Annals of Applied Probability},
    VOLUME = {29},
      YEAR = {2019},
    NUMBER = {3},
     PAGES = {1685--1739},
}
@Article{DM80,
  author ={ C. Dellacherie and P.A. Meyer},
  title ={Probabilit\'es et Potentiel. Chap V-VIII},
  journal ={ Herman, Paris},
  year = {1980},
  OPTkey =   {},
  OPTvolume = {},
  OPTnumber = {},
  pages =     {},
  OPTmonth =     {},
  OPTnote =      {},
  OPTannote =    {}
}
@article{D16,
  title={BSDEs with nonlinear weak terminal condition},
  author={Dumitrescu, Roxana},
  journal={Preprint arXiv:1602.00321},
  year={2016}
}

@article{DEH19,
  title={Mean-field reflected backward stochastic differential equations},
  author={Djehiche, B. and Elie, R. and Hamadene, S.},
  journal={To appear in AAP.Preprint arXiv:1911.06079},
  year={2019}
}

@book {Y95,
    AUTHOR = {Yeh, J.},
     TITLE = {Martingales and stochastic analysis},
    SERIES = {Series on Multivariate Analysis},
    VOLUME = {1},
 PUBLISHER = {World Scientific Publishing Co., Inc., River Edge, NJ},
      YEAR = {1995},
     PAGES = {xiv+501},
      ISBN = {981-02-2477-X},
   MRCLASS = {60-02 (60G44 60H05 60H10)},
  MRNUMBER = {1412800},
MRREVIEWER = {Vigirdas Mackevi\"Aius},
       DOI = {10.1142/9789812779304},
       URL = {http://dx.doi.org/10.1142/9789812779304},
}
 
@article{BEH16, 

    AUTHOR = {Briand, Philippe and Elie, Romuald and Hu, Ying},
     TITLE = {B{SDE}s with mean reflection},
   JOURNAL = {Ann. Appl. Probab.},
  FJOURNAL = {The Annals of Applied Probability},
    VOLUME = {28},
      YEAR = {2018},
    NUMBER = {1},
     PAGES = {482--510},
      ISSN = {1050-5164},
   MRCLASS = {60H10 (91G10)},
  MRNUMBER = {3770882},
       DOI = {10.1214/17-AAP1310},
       URL = {https://doi.org/10.1214/17-AAP1310},
}
	
@article {BV10,
    AUTHOR = {Bouchard, Bruno and Vu, Thanh Nam},
     TITLE = {The obstacle version of the geometric dynamic programming
              principle: application to the pricing of {A}merican options
              under constraints},
   JOURNAL = {Appl. Math. Optim.},
  FJOURNAL = {Applied Mathematics and Optimization},
    VOLUME = {61},
      YEAR = {2010},
    NUMBER = {2},
     PAGES = {235--265},
      ISSN = {0095-4616},
   MRCLASS = {91B25 (90C15 90C39 91G20 93E20)},
  MRNUMBER = {2585143},
MRREVIEWER = {Juan Li},
       DOI = {10.1007/s00245-009-9084-y},
       URL = {http://dx.doi.org/10.1007/s00245-009-9084-y},
}

@article{BBC16,
  title={A backward dual representation for the quantile hedging of Bermudan options},
  author={Bouchard, Bruno and Bouveret, G{\'e}raldine and Chassagneux, Jean-Fran{\c{c}}ois},
  journal={SIAM Journal on Financial Mathematics},
  volume={7},
  number={1},
  pages={215--235},
  year={2016},
  publisher={SIAM}
}

@article{BPT16,
  title={A general Doob-Meyer-Mertens decomposition for g-supermartingale systems},
  author={Bouchard, Bruno and Possama{\"\i}, Dylan and Tan, Xiaolu},
  journal={Electronic Journal of Probability},
  volume={21},
  year={2016},
  publisher={The Institute of Mathematical Statistics and the Bernoulli Society}
}

@article{STZ12,
	Author = {Soner, H. M. and Touzi, N. and Zhang, J.},
	Journal = {Probab. Theory Related Fields},
	Number = {1-2},
	Pages = {149--190},
	Title = {Well-posedness of second order backward {SDE}s},
	Volume = {153},
	Year = {2012}}
	
@article{STZ13,
	Author = {Soner, H. M. and Touzi, N. and Zhang, J.},
	Fjournal = {The Annals of Applied Probability},
	Journal = {Ann. Appl. Probab.},
	Number = {1},
	Pages = {308--347},
	Title = {Dual formulation of second order target problems},
	Volume = {23},
	Year = {2013}}	
@article {Mc63,
    AUTHOR = {McKean, Jr., H. P.},
     TITLE = {A. {S}korohod's stochastic integral equation for a reflecting
              barrier diffusion},
   JOURNAL = {J. Math. Kyoto Univ.},
  FJOURNAL = {Journal of Mathematics of Kyoto University},
    VOLUME = {3},
      YEAR = {1963},
     PAGES = {85--88},
      ISSN = {0023-608X},
   MRCLASS = {60.75 (60.62)},
  MRNUMBER = {0157406 (28 \#640)},
MRREVIEWER = {D. L. Hanson},
}
@article {BET10,
    AUTHOR = {Bouchard, Bruno and Elie, Romuald and Touzi, Nizar},
     TITLE = {Stochastic target problems with controlled loss},
   JOURNAL = {SIAM J. Control Optim.},
  FJOURNAL = {SIAM Journal on Control and Optimization},
    VOLUME = {48},
      YEAR = {2009/10},
    NUMBER = {5},
     PAGES = {3123--3150},
      ISSN = {0363-0129},
     CODEN = {SJCODC},
   MRCLASS = {49L25 (35R60 60J60 91G80 93E20)},
  MRNUMBER = {2599913},
MRREVIEWER = {Monica Motta},
       DOI = {10.1137/08073593X},
       URL = {http://dx.doi.org/10.1137/08073593X},
}
@article {FL99,
    AUTHOR = {F{\"o}llmer, Hans and Leukert, Peter},
     TITLE = {Quantile hedging},
   JOURNAL = {Finance Stoch.},
  FJOURNAL = {Finance and Stochastics},
    VOLUME = {3},
      YEAR = {1999},
    NUMBER = {3},
     PAGES = {251--273},
      ISSN = {0949-2984},
   MRCLASS = {91B30 (60H30 62F03 62P05 91B28)},
  MRNUMBER = {1842286},
MRREVIEWER = {Volkert Paulsen},
       DOI = {10.1007/s007800050062},
       URL = {http://dx.doi.org/10.1007/s007800050062},
}
@article{K95,
	Author = {R. Karandikar},
	Journal = {Stochastic Processes and Their Applications, 57:11-18},
	Title = {On pathwise stochastic integration},
	Year = {1995}}

@article {ST02,
    AUTHOR = {Soner, H. Mete and Touzi, Nizar},
     TITLE = {Stochastic target problems, dynamic programming, and viscosity
              solutions},
   JOURNAL = {SIAM J. Control Optim.},
  FJOURNAL = {SIAM Journal on Control and Optimization},
    VOLUME = {41},
      YEAR = {2002},
    NUMBER = {2},
     PAGES = {404--424},
      ISSN = {0363-0129},
     CODEN = {SJCODC},
   MRCLASS = {49K45 (49J40 49L25 91B28 93E20)},
  MRNUMBER = {1920265},
MRREVIEWER = {Fausto Gozzi},
       DOI = {10.1137/S0363012900378863},
       URL = {http://dx.doi.org/10.1137/S0363012900378863},
}
@article {STZ11,
    AUTHOR = {Soner, H. M. and Touzi, N. and Zhang, J.},
     TITLE = {Quasi-sure stochastic analysis through aggregation},
   JOURNAL = {Electron. J. Probab.},
    VOLUME = {16},
      YEAR = {2011},
      number ={67},
     PAGES = {1844--1879},
}

@article {ST03,
    AUTHOR = {Soner, H. Mete and Touzi, Nizar},
     TITLE = {A stochastic representation for mean curvature type geometric flows},
   JOURNAL = {Ann. Probab.},
  FJOURNAL = {The Annals of Probability},
    VOLUME = {31},
      YEAR = {2003},
    NUMBER = {3},
     PAGES = {1145--1165},
      ISSN = {0091-1798},
   MRCLASS = {60J60 (35K55 49Q20 58J65 60H10 74N99 91B28)},
  MRNUMBER = {1988466},
MRREVIEWER = {Olivier Raimond},
       DOI = {10.1214/aop/1055425773},
       URL = {http://dx.doi.org/10.1214/aop/1055425773},
}

@article {ElkChM80,
    AUTHOR = {Chaleyat-Maurel, M. and El Karoui, N. and Marchal, B.},
     TITLE = {R\'eflexion discontinue et syst\`emes stochastiques},
   JOURNAL = {Ann. Probab.},
  FJOURNAL = {The Annals of Probability},
    VOLUME = {8},
      YEAR = {1980},
    NUMBER = {6},
     PAGES = {1049--1067},
      ISSN = {0091-1798},
     CODEN = {APBYAE},
   MRCLASS = {60H20 (60G44 60J50)},
  MRNUMBER = {602379 (82k:60135)},
MRREVIEWER = {M. H. A. Davis},
       URL =
              {http://links.jstor.org/sici?sici=0091-1798(198012)8:6<1049:RDESS>2.0.CO;2-I&origin=MSN},
}
@incollection {P97,
    AUTHOR = {Peng, S.},
     TITLE = {Backward {SDE} and related {$g$}-expectation},
 BOOKTITLE = {Backward stochastic differential equations ({P}aris,
              1995--1996)},
    SERIES = {Pitman Res. Notes Math. Ser.},
    VOLUME = {364},
     PAGES = {141--159},
 PUBLISHER = {Longman, Harlow},
      YEAR = {1997},
   MRCLASS = {60A10 (60H10 91B16)},
  MRNUMBER = {1752680},
}

@article{KKPPQ97,
	Author = {N. {\uppercase{e}l} Karoui and C. Kapoudjian and E. Pardoux and S. Peng and M.C. Quenez},
	Journal = {Annals of Probability},
	Number = {25},
	Pages = {702-737},
	Title = {Reflected Solutions of Backward SDE and Related Obstacle Problems for PDE},
	Volume = {},
	Year = {1997}
}
@article{K82,
	Author = {H. Kunita },
	Journal = {Ecole d'\'{e}t\'{e} de Probabilit\'{e} de Saint-Flour,  Lect. Notes Math.},
	Number = {},
	Pages = {144-303},
	Title = {Stochastic differential equations and stochastic flows of diffeomorphisms},
	Volume = {1097},
	Year = {1982}
}

@article{BPTZ15,
author = {Bruno Bouchard and Dylan Possama{\"\i} and Xiaolu Tan},

title = {{A general Doob-Meyer-Mertens decomposition for g-supermartingale systems}},
volume = {21},
journal = {Electronic Journal of Probability},
number = {},
publisher = {Institute of Mathematical Statistics and Bernoulli Society},
pages = {1 -- 21},
keywords = {Backward stochastic differential equations, Doob-Meyer decomposition, non-linear expectations},
year = {2016},
doi = {10.1214/16-EJP4527},
URL = {https://doi.org/10.1214/16-EJP4527}
}

@article{M72,
  title={Th{\'e}orie des processus stochastiques g{\'e}n{\'e}raux applications aux surmartingales},
  author={Jean Mertens},
  journal={Zeitschrift f{\"u}r Wahrscheinlichkeitstheorie und Verwandte Gebiete},
  year={1972},
  volume={22},
  pages={45-68}
}

@article{PP92,
	Author = {E. Pardoux and S. Peng },
	Journal = {Lect. Notes Control Inf. Sci.},
	Number = {},
	Pages = {200-217 },
	Title = {Backward stochastic differential equations and quasilinear parabolic partial equations},
	Volume = {176},
	Year = {1992}
}

@article {PP90,
    AUTHOR = {Pardoux, {\'E}. and Peng, S. G.},
     TITLE = {Adapted solution of a backward stochastic differential
              equation},
   JOURNAL = {Systems Control Lett.},
  FJOURNAL = {Systems \& Control Letters},
    VOLUME = {14},
      YEAR = {1990},
    NUMBER = {1},
     PAGES = {55--61},
      ISSN = {0167-6911},
     CODEN = {SCLEDC},
   MRCLASS = {60H10 (60H20 93E03)},
  MRNUMBER = {1037747 (91e:60171)},
MRREVIEWER = {Kiyomasa Narita},
       DOI = {10.1016/0167-6911(90)90082-6},
       URL = {http://dx.doi.org/10.1016/0167-6911(90)90082-6},
}

@article{PTZ15,
	Author = {Possama{\"\i}, D. and Tan, X. and Zhou, C.},
	Date = {2015},
	Date-Modified = {2015-08-22 10:26:07 +0000},
	Journal = {	arXiv:1510.08439},
	Title = {Stochastic control for a class of non-linear stochastic kernels and applications}}
@article {MPZ15,
    AUTHOR = {Matoussi, Anis and Possama{\"{\i}}, Dylan and Zhou, Chao},
     TITLE = {Robust utility maximization in nondominated models with
              2{BSDE}: the uncertain volatility model},
   JOURNAL = {Math. Finance},
  FJOURNAL = {Mathematical Finance. An International Journal of Mathematics,
              Statistics and Financial Economics},
    VOLUME = {25},
      YEAR = {2015},
    NUMBER = {2},
     PAGES = {258--287},
      ISSN = {0960-1627},
   MRCLASS = {91G80},
  MRNUMBER = {3321250},
MRREVIEWER = {Wanyi Chen},
       DOI = {10.1111/mafi.12031},
       URL = {http://dx.doi.org/10.1111/mafi.12031},
}
@article {CK96,
    AUTHOR = {Cvitani{\'c}, Jak{\v{s}}a and Karatzas, Ioannis},
     TITLE = {Backward stochastic differential equations with reflection and
              {D}ynkin games},
   JOURNAL = {Ann. Probab.},
  FJOURNAL = {The Annals of Probability},
    VOLUME = {24},
      YEAR = {1996},
    NUMBER = {4},
     PAGES = {2024--2056},
      ISSN = {0091-1798},
     CODEN = {APBYAE},
   MRCLASS = {93E05 (34F05 60G40 60H10)},
  MRNUMBER = {1415239},
MRREVIEWER = {{\L}. Stettner},
       DOI = {10.1214/aop/1041903216},
       URL = {http://dx.doi.org/10.1214/aop/1041903216},
}

@article {Peng1991,
    AUTHOR = {Peng, Shi Ge},
     TITLE = {Probabilistic interpretation for systems of quasilinear
              parabolic partial differential equations},
   JOURNAL = {Stochastics Stochastics Rep.},
  FJOURNAL = {Stochastics and Stochastics Reports},
    VOLUME = {37},
      YEAR = {1991},
    NUMBER = {1-2},
     PAGES = {61--74},
      ISSN = {1045-1129},
     CODEN = {STOCBS},
   MRCLASS = {35R60 (35K50 60H30 93E20)},
  MRNUMBER = {1149116 (93a:35159)},
MRREVIEWER = {Krystyna Twardowska},
}

@article {BET09,
    AUTHOR = {Bouchard, Bruno and Elie, Romuald and Touzi, Nizar},
     TITLE = {Stochastic target problems with controlled loss},
   JOURNAL = {SIAM J. Control Optim.},
  FJOURNAL = {SIAM Journal on Control and Optimization},
    VOLUME = {48},
      YEAR = {2009/10},
    NUMBER = {5},
     PAGES = {3123--3150},
      ISSN = {0363-0129},
     CODEN = {SJCODC},
   MRCLASS = {49L25 (35R60 60J60 91G80 93E20)},
  MRNUMBER = {2599913 (2011e:49039)},
MRREVIEWER = {Monica Motta},
       DOI = {10.1137/08073593X},
       URL = {http://dx.doi.org/10.1137/08073593X},
}

@ARTICLE {TA09,
title = {Robust efficient hedging for American options: The existence of worst case probability measures},
author = {Trevi{\~n}o Aguilar, Erick},
year = {2009},
journal = {Statistics & Risk Modeling},
volume = {27},
number = {1},
pages = {1-23},
}

@article{BER15,
 ISSN = {00911798},
 URL = {http://www.jstor.org/stable/24519154},
 abstract = {We introduce a new class of backward stochastic differential equations in which the T-terminal value YT of the solution (Y, Z) is not fixed as a random variable, but only satisfies a weak constraint of the form E[Ψ(YT)] ≥ m, for some (possibly random) nondecreasing map Ψ and some threshold m. We name them BSDEs with weak terminal condition and obtain a representation of the minimal time t-values Yt such that (Y, Z) is a supersolution of the BSDE with weak terminal condition. It provides a non-Markovian BSDE formulation of the PDE characterization obtained for Markovian stochastic target problems under controlled loss in Bouchard, Elie and Touzi [SIAM J. Control Optim. 48 (2009/10) 3123–3150]. We then study the main properties of this minimal value. In particular, we analyze its continuity and convexity with respect to the m-parameter appearing in the weak terminal condition, and show how it can be related to a dual optimal control problem in Meyer form. These last properties generalize to a non-Markovian framework previous results on quantile hedging and hedging under loss constraints obtained in Föllmer and Leukert [Finance Stoch. 3 (1999) 251–273; Finance Stoch. 4 (2000) 117–146], and in Bouchard, Elie and Touzi (2009/10).},
 author = {Bruno Bouchard and Romuald Elie and Antony Réveillac},
 journal = {The Annals of Probability},
 number = {2},
 pages = {572--604},
 publisher = {Institute of Mathematical Statistics},
 title = {BSDES WITH WEAK TERMINAL CONDITION},
 volume = {43},
 year = {2015}
}
@article {TA16,
    AUTHOR = {Trevi{\~n}o Aguilar, Erick},
     TITLE = {Partial hedging of {A}merican options in discrete time and
              complete markets: convex duality and optimal {M}arkov
              policies},
   JOURNAL = {Bol. Soc. Mat. Mex. (3)},
  FJOURNAL = {Bolet\'\i n de la Sociedad Matem\'atica Mexicana. Third
              Series},
    VOLUME = {22},
      YEAR = {2016},
    NUMBER = {1},
     PAGES = {281--308},
      ISSN = {1405-213X},
   MRCLASS = {91G20 (60H30 60J05 90C15 90C25 91G80)},
  MRNUMBER = {3473762},
       DOI = {10.1007/s40590-015-0070-x},
       URL = {http://dx.doi.org/10.1007/s40590-015-0070-x},
}

@article{BEM18,
  title={Regularity of BSDEs with a convex constraint on the gains-process},
  author={Bouchard, Bruno and Elie, Romuald and Moreau, Ludovic},
  journal={Bernoulli},
  volume={24},
  number={3},
  pages={1613--1635},
  year={2018},
  publisher={Bernoulli Society for Mathematical Statistics and Probability}
}
@article{PX10,
author = {Shige Peng and Mingyu Xu},
title = {{Reflected BSDE with a constraint and its applications in an incomplete market}},
volume = {16},
journal = {Bernoulli},
number = {3},
publisher = {Bernoulli Society for Mathematical Statistics and Probability},
pages = {614 -- 640},
keywords = {American options in an incomplete market, backward stochastic differential equation with a constraint, reflected backward stochastic differential equation},
year = {2010},
doi = {10.3150/09-BEJ227},
URL = {https://doi.org/10.3150/09-BEJ227}
}
@article{DHK13,
author = {Samuel Drapeau and Gregor Heyne and Michael Kupper},
title = {{Minimal supersolutions of convex BSDEs}},
volume = {41},
journal = {The Annals of Probability},
number = {6},
publisher = {Institute of Mathematical Statistics},
pages = {3973 -- 4001},
keywords = {nonlinear expectations, Supermartingales, Supersolutions of backward stochastic differential equations},
year = {2013},
doi = {10.1214/13-AOP834},
URL = {https://doi.org/10.1214/13-AOP834}
}
@article{HKM14,
author = {Gregor Heyne and Michael Kupper and Christoph Mainberger},
title = {{Minimal supersolutions of BSDEs with lower semicontinuous generators}},
volume = {50},
journal = {Annales de l'Institut Henri Poincaré, Probabilités et Statistiques},
number = {2},
publisher = {Institut Henri Poincaré},
pages = {524 -- 538},
keywords = {Semimartingale convergence, Supersolutions of backward stochastic differential equations},
year = {2014},
doi = {10.1214/12-AIHP523},
URL = {https://doi.org/10.1214/12-AIHP523}
}
@article{KLIMSIAK15,
title = {Reflected BSDEs on filtered probability spaces},
journal = {Stochastic Processes and their Applications},
volume = {125},
number = {11},
pages = {4204-4241},
year = {2015},
issn = {0304-4149},
doi = {https://doi.org/10.1016/j.spa.2015.06.006},
url = {https://www.sciencedirect.com/science/article/pii/S0304414915001556},
author = {Tomasz Klimsiak},
keywords = {Reflected BSDE, General filtration,  data},
abstract = {We study the problem of existence and uniqueness of solutions of backward stochastic differential equations with two reflecting irregular barriers, Lp data and generators satisfying weak integrability conditions. We deal with equations on general filtered probability spaces. In case the generator does not depend on the z variable, we first consider the case p=1 and we only assume that the underlying filtration satisfies the usual conditions of right-continuity and completeness. Additional integrability properties of solutions are established if pâ(1,2] and the filtration is quasi-continuous. In case the generator depends on z, we assume that p=2, the filtration satisfies the usual conditions and additionally that it is separable. Our results apply for instance to Markov-type reflected backward equations driven by general Hunt processes.}
}
@article {HTA13,
    AUTHOR = {P{\'e}rez Hern{\'a}ndez, Leonel and Trevi{\~n}o Aguilar,
              Erick},
     TITLE = {\frac{â¢}{â¢}},
   JOURNAL = {Bol. Soc. Mat. Mexicana (3)},
  FJOURNAL = {Sociedad Matem\'atica Mexicana. Bolet\'\i n. Tercera Serie},
    VOLUME = {19},
      YEAR = {2013},
    NUMBER = {2},
     PAGES = {237--253},
      ISSN = {1405-213X},
   MRCLASS = {91B30 (60H30)},
  MRNUMBER = {3183995},
MRREVIEWER = {Hong Miao},
}
@article {BHZ15,
    AUTHOR = {Bayraktar, Erhan and Huang, Yu-Jui and Zhou, Zhou},
     TITLE = {On hedging {A}merican options under model uncertainty},
   JOURNAL = {SIAM J. Financial Math.},
  FJOURNAL = {SIAM Journal on Financial Mathematics},
    VOLUME = {6},
      YEAR = {2015},
    NUMBER = {1},
     PAGES = {425--447},
      ISSN = {1945-497X},
   MRCLASS = {91G20 (49L20 60G40 60G42 91G80 93E20)},
  MRNUMBER = {3356981},
MRREVIEWER = {Jos{\'e} Fajardo},
       DOI = {10.1137/140961869},
       URL = {http://dx.doi.org/10.1137/140961869},
}
@article {NZ15,
    AUTHOR = {Nutz, Marcel and Zhang, Jianfeng},
     TITLE = {Optimal stopping under adverse nonlinear expectation and
              related games},
   JOURNAL = {Ann. Appl. Probab.},
  FJOURNAL = {The Annals of Applied Probability},
    VOLUME = {25},
      YEAR = {2015},
    NUMBER = {5},
     PAGES = {2503--2534},
      ISSN = {1050-5164},
   MRCLASS = {60G40 (91A15 91A60 91G20)},
  MRNUMBER = {3375882},
MRREVIEWER = {Michael Ludkovski},
       DOI = {10.1214/14-AAP1054},
       URL = {http://dx.doi.org/10.1214/14-AAP1054},
}
	
@article {FL00,
    AUTHOR = {F{\"o}llmer, Hans and Leukert, Peter},
     TITLE = {Efficient hedging: cost versus shortfall risk},
   JOURNAL = {Finance Stoch.},
  FJOURNAL = {Finance and Stochastics},
    VOLUME = {4},
      YEAR = {2000},
    NUMBER = {2},
     PAGES = {117--146},
      ISSN = {0949-2984},
   MRCLASS = {91B30},
  MRNUMBER = {1780323 (2001f:91054)},
MRREVIEWER = {Thomas W. Epps},
       DOI = {10.1007/s007800050008},
       URL = {http://dx.doi.org/10.1007/s007800050008},
}
@article {FL99,
    AUTHOR = {F{\"o}llmer, Hans and Leukert, Peter},
     TITLE = {Quantile hedging},
   JOURNAL = {Finance Stoch.},
  FJOURNAL = {Finance and Stochastics},
    VOLUME = {3},
      YEAR = {1999},
    NUMBER = {3},
     PAGES = {251--273},
      ISSN = {0949-2984},
   MRCLASS = {91B30 (60H30 62F03 62P05 91B28)},
  MRNUMBER = {1842286 (2002g:91096)},
MRREVIEWER = {Volkert Paulsen},
       DOI = {10.1007/s007800050062},
       URL = {http://dx.doi.org/10.1007/s007800050062},
}
@article {DK07,
    AUTHOR = {Dolinsky, Yan and Kifer, Yuri},
     TITLE = {Hedging with risk for game options in discrete time},
   JOURNAL = {Stochastics},
  FJOURNAL = {Stochastics. An International Journal of Probability and
              Stochastic Processes},
    VOLUME = {79},
      YEAR = {2007},
    NUMBER = {1-2},
     PAGES = {169--195},
      ISSN = {1744-2508},
   MRCLASS = {60G40 (91A05 91B28 91B30)},
  MRNUMBER = {2290404},
       DOI = {10.1080/17442500601097784},
       URL = {http://dx.doi.org/10.1080/17442500601097784},
}
@article {PH07,
    AUTHOR = {P\'erez-Hern\'andez, Leonel},
     TITLE = {On the existence of an efficient hedge for an {A}merican
              contingent claim within a discrete time market},
   JOURNAL = {Quant. Finance},
  FJOURNAL = {Quantitative Finance},
    VOLUME = {7},
      YEAR = {2007},
    NUMBER = {5},
     PAGES = {547--551},
      ISSN = {1469-7688},
   MRCLASS = {91B28 (60G40)},
  MRNUMBER = {2358918},
MRREVIEWER = {Michael Ludkovski},
       DOI = {10.1080/14697680601158700},
       URL = {http://dx.doi.org/10.1080/14697680601158700},
}	
@article {M11,
    AUTHOR = {Mulinacci, Sabrina},
     TITLE = {The efficient hedging problem for {A}merican options},
   JOURNAL = {Finance Stoch.},
  FJOURNAL = {Finance and Stochastics},
    VOLUME = {15},
      YEAR = {2011},
    NUMBER = {2},
     PAGES = {365--397},
      ISSN = {0949-2984},
   MRCLASS = {91G20 (49L20 60G40 60H30 93E20)},
  MRNUMBER = {2800220},
MRREVIEWER = {Ingo Fahrner},
       DOI = {10.1007/s00780-010-0151-7},
       URL = {http://dx.doi.org/10.1007/s00780-010-0151-7},
}				

@article {M11,
    AUTHOR = {Moreau, Ludovic},
     TITLE = {Stochastic target problems with controlled loss in jump
              diffusion models},
   JOURNAL = {SIAM J. Control Optim.},
  FJOURNAL = {SIAM Journal on Control and Optimization},
    VOLUME = {49},
      YEAR = {2011},
    NUMBER = {6},
     PAGES = {2577--2607},
      ISSN = {0363-0129},
     CODEN = {SJCODC},
   MRCLASS = {49L25 (35K55 60J75 93E20)},
  MRNUMBER = {2873197},
MRREVIEWER = {Pavel Pakshin},
       DOI = {10.1137/100802268},
       URL = {http://dx.doi.org/10.1137/100802268},
}
@article {PX05,
    AUTHOR = {Peng, Shige and Xu, Mingyu},
     TITLE = {The smallest {$g$}-supermartingale and reflected {BSDE} with
              single and double {$L^2$} obstacles},
   JOURNAL = {Ann. Inst. H. Poincar\'e Probab. Statist.},
  FJOURNAL = {Annales de l'Institut Henri Poincar\'e. Probabilit\'es et
              Statistiques},
    VOLUME = {41},
      YEAR = {2005},
    NUMBER = {3},
     PAGES = {605--630},
      ISSN = {0246-0203},
     CODEN = {AHPBAR},
   MRCLASS = {60G40 (60H10 60H30 60H99)},
  MRNUMBER = {2139035},
MRREVIEWER = {Sa{\"{\i}}d Hamadene},
       DOI = {10.1016/j.anihpb.2004.12.002},
       URL = {http://dx.doi.org/10.1016/j.anihpb.2004.12.002},
}
@article {DQS16,
    AUTHOR = {Dumitrescu, Roxana and Quenez, Marie-Claire and Sulem, Agn\`es},
     TITLE = {A weak dynamic programming principle for combined optimal
              stopping/stochastic control with {$\mathcal{E}^f$}-expectations},
   JOURNAL = {SIAM J. Control Optim.},
  FJOURNAL = {SIAM Journal on Control and Optimization},
    VOLUME = {54},
      YEAR = {2016},
    NUMBER = {4},
     PAGES = {2090--2115},
      ISSN = {0363-0129},
   MRCLASS = {60H10 (47N10 49L20 93E20)},
  MRNUMBER = {3539885},
       DOI = {10.1137/15M1027012},
       URL = {https://doi.org/10.1137/15M1027012},
}
@article {GIOOQ17,
    AUTHOR = {Grigorova, Miryana and Imkeller, Peter and
              Ouknine, Youssef and Quenez, Marie-Claire},
     TITLE = {Doubly Reflected BSDEs and Ef-Dynkin games: beyond
the right-continuous case},
   JOURNAL = {Electron. J. Probab.},
  FJOURNAL = {Electron. J. Probab.},
    VOLUME = {23},
      YEAR = {2018},
    NUMBER = {122},
     PAGES = {1-38},
      ISSN = {},
   MRCLASS = {},
  MRNUMBER = {},
       DOI = {},
       URL = {hhttps://doi.org/10.1214/18-EJP225},
}

@article {GIOOQ171,
    AUTHOR = {Grigorova, Miryana and Imkeller, Peter and Offen, Elias and
              Ouknine, Youssef and Quenez, Marie-Claire},
     TITLE = {Reflected BSDEs when the obstacle is not right-continuous and optimal stopping},
   JOURNAL = {Ann. Appl. Probab.},
  FJOURNAL = {Ann. Appl. Probab.},
    VOLUME = {27},
      YEAR = {2017},
    NUMBER = {5},
     PAGES = {3153-3188},
      ISSN = {},
   MRCLASS = {},
  MRNUMBER = {},
       DOI = {10.1214/17-AAP1278},
       URL = {},
}
\bibliographystyle{acm}
\bibliography{WRBSDE}

\begin{thebibliography}{10}

\bibitem{BBC16}
{\sc Bouchard, B., Bouveret, G., and Chassagneux, J.-F.}
\newblock A backward dual representation for the quantile hedging of bermudan
  options.
\newblock {\em SIAM Journal on Financial Mathematics 7}, 1 (2016), 215--235.

\bibitem{BEM18}
{\sc Bouchard, B., Elie, R., and Moreau, L.}
\newblock Regularity of bsdes with a convex constraint on the gains-process.
\newblock {\em Bernoulli 24}, 3 (2018), 1613--1635.

\bibitem{BER15}
{\sc Bouchard, B., Elie, R., and Réveillac, A.}
\newblock Bsdes with weak terminal condition.
\newblock {\em The Annals of Probability 43}, 2 (2015), 572--604.

\bibitem{BET10}
{\sc Bouchard, B., Elie, R., and Touzi, N.}
\newblock Stochastic target problems with controlled loss.
\newblock {\em SIAM J. Control Optim. 48}, 5 (2009/10), 3123--3150.

\bibitem{BPT16}
{\sc Bouchard, B., Possama{\"\i}, D., and Tan, X.}
\newblock A general doob-meyer-mertens decomposition for g-supermartingale
  systems.
\newblock {\em Electronic Journal of Probability 21\/} (2016).

\bibitem{BEH16}
{\sc Briand, P., Elie, R., and Hu, Y.}
\newblock B{SDE}s with mean reflection.
\newblock {\em Ann. Appl. Probab. 28}, 1 (2018), 482--510.

\bibitem{CK96}
{\sc Cvitani{\'c}, J., and Karatzas, I.}
\newblock Backward stochastic differential equations with reflection and
  {D}ynkin games.
\newblock {\em Ann. Probab. 24}, 4 (1996), 2024--2056.

\bibitem{DL82}
{\sc Dellacherie, C., and Lenglart, E.}
\newblock Sur des probl\`emes de r\'egularisation, de recollement et
  d'interpolation en th\'eorie des processus.
\newblock In {\em Seminar on {P}robability, {XVI}}, vol.~920 of {\em Lecture
  Notes in Math.} Springer, Berlin, 1982, pp.~298--313.

\bibitem{DM80}
{\sc Dellacherie, C., and Meyer, P.}
\newblock Probabilit\'es et potentiel. chap v-viii.
\newblock {\em Herman, Paris\/} (1980).

\bibitem{DHK13}
{\sc Drapeau, S., Heyne, G., and Kupper, M.}
\newblock {Minimal supersolutions of convex BSDEs}.
\newblock {\em The Annals of Probability 41}, 6 (2013), 3973 -- 4001.

\bibitem{D16}
{\sc Dumitrescu, R.}
\newblock Bsdes with nonlinear weak terminal condition.
\newblock {\em Preprint arXiv:1602.00321\/} (2016).

\bibitem{DQS16}
{\sc Dumitrescu, R., Quenez, M.-C., and Sulem, A.}
\newblock A weak dynamic programming principle for combined optimal
  stopping/stochastic control with {$\mathcal{E}^f$}-expectations.
\newblock {\em SIAM J. Control Optim. 54}, 4 (2016), 2090--2115.

\bibitem{GIOOQ171}
{\sc Grigorova, M., Imkeller, P., Offen, E., Ouknine, Y., and Quenez, M.-C.}
\newblock Reflected bsdes when the obstacle is not right-continuous and optimal
  stopping.
\newblock {\em Ann. Appl. Probab. 27}, 5 (2017), 3153--3188.

\bibitem{GIOOQ17}
{\sc Grigorova, M., Imkeller, P., Ouknine, Y., and Quenez, M.-C.}
\newblock Doubly reflected bsdes and ef-dynkin games: beyond the
  right-continuous case.
\newblock {\em Electron. J. Probab. 23}, 122 (2018), 1--38.

\bibitem{HKM14}
{\sc Heyne, G., Kupper, M., and Mainberger, C.}
\newblock {Minimal supersolutions of BSDEs with lower semicontinuous
  generators}.
\newblock {\em Annales de l'Institut Henri Poincaré, Probabilités et
  Statistiques 50}, 2 (2014), 524 -- 538.

\bibitem{KLIMSIAK15}
{\sc Klimsiak, T.}
\newblock Reflected bsdes on filtered probability spaces.
\newblock {\em Stochastic Processes and their Applications 125}, 11 (2015),
  4204--4241.

\bibitem{MPZ17}
{\sc Matoussi, A., Possama{\"\i}, D., and Zhou, C.}
\newblock {Corrigendum for Second-order reflected backward stochastic
  differential equations and Second-order BSDEs with general reflection and
  game options under uncertainty}.
\newblock {\em The Annals of Applied Probability 31}, 3 (2021), 1505 -- 1522.

\bibitem{M72}
{\sc Mertens, J.}
\newblock Th{\'e}orie des processus stochastiques g{\'e}n{\'e}raux applications
  aux surmartingales.
\newblock {\em Zeitschrift f{\"u}r Wahrscheinlichkeitstheorie und Verwandte
  Gebiete 22\/} (1972), 45--68.

\bibitem{N75}
{\sc Neveu, J.}
\newblock {\em Discrete-parameter Martingales}.
\newblock Mathematical Studies. North-Holland, 1975.

\bibitem{P97}
{\sc Peng, S.}
\newblock Backward {SDE} and related {$g$}-expectation.
\newblock In {\em Backward stochastic differential equations ({P}aris,
  1995--1996)}, vol.~364 of {\em Pitman Res. Notes Math. Ser.} Longman, Harlow,
  1997, pp.~141--159.

\bibitem{peng:99}
{\sc Peng, S.}
\newblock {Monotonic limit theorem of BSDE and nonlinear decomposition theorem
  of Doob-MeyerF's type.}
\newblock {\em Probab. Theory Relat. Fields 113}, 4 (1999), 473--499.

\bibitem{PX10}
{\sc Peng, S., and Xu, M.}
\newblock {Reflected BSDE with a constraint and its applications in an
  incomplete market}.
\newblock {\em Bernoulli 16}, 3 (2010), 614 -- 640.

\bibitem{PZ18}
{\sc Popier, A., and Zhou, C.}
\newblock Second order bsde under monotonicity condition and liquidation
  problem under uncertainty.
\newblock {\em Ann. Appl. Probab. 29}, 3 (2019), 1685--1739.

\bibitem{STZ12}
{\sc Soner, H.~M., Touzi, N., and Zhang, J.}
\newblock Well-posedness of second order backward {SDE}s.
\newblock {\em Probab. Theory Related Fields 153}, 1-2 (2012), 149--190.

\end{thebibliography}

%
%
%
%
%

\end{document}